\documentclass[final,3p,times]{elsarticle}       % 单栏

\usepackage{lineno} %,hyperref
\modulolinenumbers[5]

\makeatletter
\def\ps@pprintTitle{%
  \let\@oddhead\@empty
  \let\@evenhead\@empty
  \let\@oddfoot\@empty
  \let\@evenfoot\@oddfoot
}
\makeatother
\usepackage[table,xcdraw,svgnames]{xcolor}
\usepackage[colorlinks]{hyperref}
\AtBeginDocument{%
  \hypersetup{
    citecolor=magenta,
    linkcolor=blue,
    urlcolor=Blue}}
\usepackage[heading=true,scheme=plain]{ctex} %%汉化英文模板 的时候   只想支持中文，标题不汉化
\usepackage{mathrsfs}
\usepackage{amsfonts}   % \mathbb{} 花体
\usepackage{amssymb}   %  \complement
\usepackage{lmodern}   % 数学公式字体
\usepackage{bbm}  % 小写字母空心 \mathbbm

\usepackage{autobreak}       %  autobreak package
\allowdisplaybreaks[4]   % 使数学公式自动分页

 \usepackage[capitalise]{cleveref}%\usepackage[nameinlink]{cleveref} %  宏包的\cref 类似于\cite引用多个文献

  \usepackage{extarrows}  % 定义箭头  \xlongequal  \xLongleftrightarrow
\usepackage{tikz}   % 绘图
\usetikzlibrary{arrows.meta}
\usetikzlibrary{snakes}
\usepackage{pgfplots}
\usetikzlibrary{patterns}
\usepackage{tkz-euclide}
\usetikzlibrary{calc}
\usetikzlibrary{intersections}
\usepackage{subfigure}       % 插图
\numberwithin{figure}{section}
\usepackage[justification=centering]{caption}  % 标题换行
 \usepackage[shortlabels]{enumitem}  % 设定排列环境的 格式, 可以缩写简略形式
\numberwithin{equation}{section} %公式按章节编号
 \allowdisplaybreaks[4]   % 使数学公式自动分页
\usepackage{amsthm}
\newtheorem{definition}{Definition} [section]            % 斜体
\newtheorem{theorem}{Theorem}[section]            % 斜体
\newtheorem{lemma}{Lemma} [section]%单独编号,与节有关
\newtheorem{proposition}{Proposition}[section]
\newdefinition{corollary}{Corollary}[section]
\newtheorem{remark}{Remark}
\def\loc{{\mathrm{loc}}}

\def\dchi{\scalebox{1.2}{$\chi$}}

\def\dint{\displaystyle\int}

\def\trace{{\rm{trace}}~}          % 迹 函数

\DeclareMathOperator*{\bmo}{BMO}

\DeclareMathOperator*{\essinf}{ess\, inf}
\DeclareMathOperator*{\esssup}{ess\, sup}
\DeclareMathOperator*{\lip}{Lip}

\DeclareMathOperator*{\almostevery}{a.e.~}

\newcommand{\mathd}{\mathrm{d}}

\biboptions{numbers,sort&compress}   % 针对natbib，引用参考文献 缩写
\begin{document}

\begin{frontmatter}

\title{{\bfseries Necessary and sufficient conditions for boundedness of commutators of fractional maximal function in variable Lebesgue spaces on stratified groups }}
% Characterizations of Lipschitz space via commutators of   fractional  maximal function in variable Lebesgue spaces  on stratified  groups

\author[mymainaddress]{Wenjiao Zhao} %\corref{mycorrespondingauthor}
%\ead{zhaowenjiao@mail.ahpu.edu.cn}
\address[mymainaddress]{School of Mathematics, Physics and Finance, Anhui Polytechnic University, Wuhu 241000, China}

\author[mysecondaryaddress]{Jianglong Wu\corref{mycorrespondingauthor}}
\cortext[mycorrespondingauthor]{Corresponding author}
 \ead{jl-wu@163.com}
\address[mysecondaryaddress]{Department of Mathematics, Mudanjiang Normal University, Mudanjiang  157011, China}

\begin{abstract}
In this paper, the main aim is to consider the   mapping properties of the maximal or nonlinear commutator for the fractional maximal operator with the symbols belong to the Lipschitz spaces on variable Lebesgue spaces in the context of  stratified Lie group, with the help of  which some new characterizations to the Lipschitz spaces and nonnegative Lipschitz functions are obtained in the stratified groups context. Meanwhile, Some equivalent relations between the Lipschitz norm and the variable Lebesgue norm are also given.
\end{abstract}


\end{frontmatter}

%--------------
%\linenumbers    % 显示行数
%-------------

%==========================
\section{Introduction and main results}
\label{sec:introduction}

Nowadays, the study of function spaces arising in the context of groups has received increasing attention.
%Nowadays, more and more attention has been paid to the study of function spaces which arise in the context of groups,
such as variable Lebesgue spaces\cite{liu2019multilinear,liu2022characterisation}, Orlicz spaces \cite{guliyev2022some} and generalized Morrey spaces \cite{guliyev2021commutators} et al.
Especially, the theory of variable exponent function spaces  has been intensely investigated in the past twenty years since some elementary properties were established by Kov{\'a}{\v{c}}ik  and R{\'a}kosn{\'\i}k in \cite{kovavcik1991spaces}.
It is remarkable that the  fractional maximal operator plays an  instrumental role in harmonic analysis and application areas, such as partial differential equations (PDEs) and potential theory,  since it  is intimately related to the Riesz potential operator, which  is a powerful tool in the investigation of  the smooth function spaces  (see \cite{folland1982hardy,bonfiglioli2007stratified,carneiro2017derivative}).
In addition, the study of commutators has attracted widespread attention, because it is not only closely related to the regularity of solutions for some PDEs \cite{chiarenza1993w2,difazio1993interior,ragusa2004cauchy,bramanti1995commutators}, but also can give some characterizations of   function spaces \cite{janson1978mean,paluszynski1995characterization}.
Meanwhile, many harmonic analysis problems  on stratified Lie groups deserve a further investigation since  most results   of the theory of distribution functions
and Fourier transforms in Euclidean spaces  cannot yet be replicated \cite{liu2019multilinear,guliyev2022some}.

Let $T$ be the classical singular integral operator. The Coifman-Rochberg-Weiss type commutator $[b, T]$ generated by $T$ and a
suitable function $b$ is defined by
\begin{align} \label{equ:commutator-1}
 [b,T]f      & = bT(f)-T(bf).
\end{align}
A well-known result shows that  $[b,T]$ is bounded on $L^{p}(\mathbb{R}^{n})$  for $1<p<\infty$ if and only if $b\in \bmo(\mathbb{R}^{n})$ (the space of bounded mean oscillation functions). The sufficiency was provided  by Coifman et al.\cite{coifman1976factorization} and the necessity was obtained by Janson \cite{janson1978mean}.
Furthermore, Janson \cite{janson1978mean} also established some characterizations of the Lipschitz space $\Lambda_{\beta}(\mathbb{R}^{n})$ via commutator \labelcref{equ:commutator-1} and  proved that  $[b,T]$ is bounded from $L^{p}(\mathbb{R}^{n})$ to $L^{q}(\mathbb{R}^{n})$   for $1<p<n/\beta$ and $1/p-1/q=\beta/n$ with $0<\beta<1$ if and only if  $b\in \Lambda_{\beta}(\mathbb{R}^{n})$  (see also Paluszy\'{n}ski \cite{paluszynski1995characterization}).

Denote by $\mathbb{G}$  and $\mathbb{R}$ the sets of   groups and   real numbers separately. Let   $Q$  be the homogeneous dimension of $\mathbb{G}$, $0\le \alpha <Q$ and $f: \mathbb{G} \to \mathbb{R}$ be a locally integrable function,  the fractional maximal function is given by
%------------------
\begin{align*}
%-----------
   \mathcal{M}_{\alpha}(f)(x)    &=  \sup_{B\ni x  \atop B\subset \mathbb{G}}  \dfrac{1}{|B|^{1-\alpha/Q}} \dint_{B}  |f(y)| \mathd y,
%-----------------
\end{align*}
%--------------
where the supremum is taken over all $ \mathbb{G}$-balls $B\subset \mathbb{G}$  containing $x$ with radius $r>0$, and $|B|$   represents  the Haar measure of the $ \mathbb{G}$-ball $B$ (for the notations and notions, see  \cref{sec:preliminary} below).
When $\alpha=0$, we simply write  $\mathcal{M} $ instead of $\mathcal{M}_{0}$, which is the  Hardy-Littlewood maximal function defined as
%----------------
\begin{align*}   %\label{equ:p-adic-max-operator}
%-------
  \mathcal{M}(f)(x)    &=  \sup_{B\ni x\atop B\subset \mathbb{G}}  \dfrac{1}{|B|} \dint_{B}  |f(y)| \mathd y.
%----------
\end{align*}
%------------

%The reader can refer to Stein \cite{stein1993harmonic} for the definition on the Euclidean case.

Based on  \labelcref{equ:commutator-1},  we  give two different kinds of commutator for the fractional maximal function as follows.
%--------------
\begin{definition} \label{def.commutator-frac-max}
%-----------------
 Let  $0 \le \alpha<Q$ and $b\in L_{\loc}^{1}(\mathbb{G})$. % $b$ be  a locally integrable function on $\mathbb{G}$.
%----------
\begin{enumerate}[label=(\roman*)]  %,align=left
%------------
\item The maximal commutator of $\mathcal{M}_{\alpha}$ with $b$ is given by
%----------
\begin{align*}
%-------
 \mathcal{M}_{\alpha,b} (f)(x) &= \sup_{B\ni x \atop B\subset \mathbb{G}}  \dfrac{1}{|B|^{1-\alpha/Q}} \dint_{B} |b(x)-b(y)| |f(y)| \mathd y,
%----------
\end{align*}
%------------
where the supremum is taken over all $ \mathbb{G}$-balls $B\subset\mathbb{G}$ containing $x$.
%-----------
 \item  The nonlinear commutators generated by  $\mathcal{M}_{\alpha}$ and $b$  is defined by
%-------
\begin{align*}
%-------
[b,\mathcal{M}_{\alpha}] (f)(x) &= b(x) \mathcal{M}_{\alpha} (f)(x) -\mathcal{M}_{\alpha}(bf)(x).
%----------
\end{align*}
%-------
\end{enumerate}
%--------------
\end{definition}
%------------

When $\alpha=0$, we simply represent by $[b,\mathcal{M}]=[b,\mathcal{M}_{0}]$ and $\mathcal{M}_{b} =\mathcal{M}_{0,b} $.

Although $[b,T]$ is a linear operator, since $[b,\mathcal{M}_{\alpha}] $ is not even a sublinear operator, it  is said to be a  nonlinear commutator.
 It is noteworthy that the maximal commutator $\mathcal{M}_{\alpha,b}$ and the nonlinear commutator $[b,\mathcal{M}_{\alpha}] $   are  inherently different from each other.
For  instance, $\mathcal{M}_{\alpha,b}$ is not only positive but also   sublinear, while $[b,\mathcal{M}_{\alpha}] $ is both not positive and not sublinear.

%We call $[b,\mathcal{M}_{\alpha}] $ the nonlinear commutator because it is not even a sublinear operator,
%although the commutator $[b,T]$ is a linear one. It is worth noting that the nonlinear commutator $[b,\mathcal{M}_{\alpha}] $ and the maximal commutator $\mathcal{M}_{\alpha,b}$ essentially differ from each other.
%For example, $\mathcal{M}_{\alpha,b}$ is positive and sublinear, but $[b,\mathcal{M}_{\alpha}] $ is neither positive nor sublinear.

Denote by $M$  and $M_{\alpha}$ the classical Hardy–Littlewood maximal function and the   fractional maximal function in $\mathbb{R}^{n}$, respectively.
 In  1990,   Milman and Schonbek\cite{milman1990second}  used the real interpolation technique to establish a commutator result  that applies to the Hardy-Littlewood maximal function as well as to a large class of nonlinear operators.
 In 2000, Bastero et al.\cite{bastero2000commutators}    proved the necessary and sufficient condition for the boundedness of  the  nonlinear    commutators $[b,M]$ and $[b,M^{\sharp}]$ on $L^{p}$ spaces.
 Zhang and Wu obtained similar results for the fractional maximal function in \cite{zhang2009commutators} and extended the mentioned results to variable exponent Lebesgue spaces in \cite{zhang2014commutatorsfor,zhang2014commutators}.
When the symbol $b$ belongs to Lipschitz spaces,   Zhang et al. \cite{zhang2018commutators,zhang2019some} gave the necessary and sufficient conditions for the boundedness of  $[b,M_{\alpha}]$ on Orlicz spaces and variable Lebesgue spaces, respectively.
Recently, Guliyev \cite{guliyev2022some,guliyev2023characterizations}
extended the mentioned results to  Orlicz spaces $L^{\Phi} (\mathbb{G})$ over some stratified Lie group when the symbols belong to $\bmo(\mathbb{G})$ spaces  and  Lipschitz spaces $\Lambda_{\beta}(\mathbb{G})$ respectively, and obtained separately some   characterizations for certain subclasses  of $\bmo(\mathbb{G})$ and $\Lambda_{\beta}(\mathbb{G})$.
And Liu et al.\cite{liu2019multilinear,liu2022characterisation} established the  boundedness of some commutators   in variable Lebesgue spaces over  stratified groups. Meanwhile, Wu and Zhao \cite{wu2023characterizationlip,wu2023some}  give some   characterizations of the
Lipschitz spaces on  stratified Lie group.

Inspired by the above literature, the purpose of this paper is to   study  the boundedness of the maximal or nonlinear commutator of the fractional maximal operator with the symbols belong to the Lipschitz spaces on variable Lebesgue spaces in the context of  stratified Lie group, by which we obtain some new characterizations of the Lipschitz spaces and nonnegative Lipschitz functions   in the stratified groups context. In addition, we give some equivalent relations between the Lipschitz norm and the variable Lebesgue norm as well.

To illuminate the results, we first introduce the following notations.

Let $\alpha\ge 0$  and $f\in L_{\loc}^{1}(\mathbb{G})$, for a fixed $\mathbb{G}$-ball $B^{*}$, the fractional maximal function with respect to $B^{*}$   is defined as
%----------------
\begin{align*}
%-------
\mathcal{M}_{\alpha,B^{*}} (f)(x) &= \sup_{ B\ni x \atop  B\subset B^{*}} \dfrac{1}{|B|^{1-\alpha/Q}} \dint_{B} |f(y)| \mathd y,
%----------
\end{align*}
%-----------
where the supremum is taken over all $\mathbb{G}$-balls $B$  such that $x\in B\subset B^{*}$.  We simply substitute $\mathcal{M}_{B^{*}}$ for $\mathcal{M}_{0,B^{*}}$ when $\alpha= 0$.

Our primary results can be  expressed as follows.

%-----------------
\begin{theorem} \label{thm:nonlinear-frac-max-lip-variable}
%---------------
 Let  $0 <\beta <1$, $0< \alpha<\alpha+\beta<Q$  and $b\in L_{\loc}^{1}(\mathbb{G})$. %$b$ be a locally integrable function on $\mathbb{G}$.
 Then the following assertions are equivalent:
%--------------
\begin{enumerate}[label=(T.\arabic*)]   %,align=left,itemindent=1em\arabic
%%--------------
\item   $b\in  \Lambda_{\beta}(\mathbb{G})$ and $b\ge 0$.
 %-----------
    \label{enumerate:thm-nonlinear-frac-max-lip-1}
%------------------
   \item The commutator $ [b,\mathcal{M}_{\alpha}] $ is bounded from $L^{p(\cdot)}(\mathbb{G})$ to $L^{q(\cdot)}(\mathbb{G})$ for all $(p(\cdot), q(\cdot))\in   \mathscr{B}^{\alpha+\beta}(\mathbb{G}) $.
%-----------
    \label{enumerate:thm-nonlinear-frac-max-lip-2}
%--------------------------------
\item  The commutator $ [b,\mathcal{M}_{\alpha}] $ is bounded from $L^{p(\cdot)}(\mathbb{G})$ to $L^{q(\cdot)}(\mathbb{G})$   for some  $(p(\cdot), q(\cdot))\in   \mathscr{B}^{\alpha+\beta}(\mathbb{G}) $.
%-----------
    \label{enumerate:thm-nonlinear-frac-max-lip-3}
%------------
   \item There exists an  $s(\cdot) \in   \mathscr{B} (\mathbb{G}) $   such that
%-----------
\begin{align} \label{inequ:thm-nonlinear-frac-max-lip-4}
%-----------
\sup_{B\subset \mathbb{G}} \dfrac{1}{|B|^{\beta/Q}} \dfrac{\Big\| \big(b -|B|^{-\alpha/Q}\mathcal{M}_{\alpha,B} (b) \big) \dchi_{B} \Big\|_{L^{s(\cdot)}(\mathbb{G}) }}{\|\dchi_{B}\|_{L^{s(\cdot)}(\mathbb{G}) }}  < \infty.
%-----------------
\end{align}
%-----------
%-----------
    \label{enumerate:thm-nonlinear-frac-max-lip-4}
%------------
   \item  \labelcref{inequ:thm-nonlinear-frac-max-lip-4} holds for all $s(\cdot) \in   \mathscr{B} (\mathbb{G}) $.
   %For all    $s(\cdot) \in   \mathscr{B} (\mathbb{G}) $, we have    \labelcref{inequ:thm-nonlinear-frac-max-lip-4} holds.
%-----------
\label{enumerate:thm-nonlinear-frac-max-lip-5}
%-----------------
\end{enumerate}
%-------------
%----------
\end{theorem}
%-----------------

%--------------
\begin{remark}   \label{rem.nonlinear-frac-max-lip}
%-----------
\begin{enumerate}[ label=(\roman*)]  %,itemindent=-0.3em,itemindent=1.5em
%-------
\item For the case $\alpha=0$,  the partial  equivalence of
\labelcref{enumerate:thm-nonlinear-frac-max-lip-1,enumerate:thm-nonlinear-frac-max-lip-2}  was proved in \cite{chang2022boundedness} (see Theorem 5.2.2).
Moreover, \labelcref{inequ:thm-nonlinear-frac-max-lip-4} gives a new characterization of nonnegative Lipschitz functions, compared with \cite[Theorem 5.2.2]{chang2022boundedness}.
%-------
%\item  Moreover, it was proved in \cite[Theorem 5.2.2]{chang2022boundedness}, see also \cref{lem:non-negative-max-lip} below, that  $b\in  \Lambda_{\beta}(\mathbb{G})$  and $b\ge 0$ if and only if
%%--------
%\begin{align} \label{inequ:thm-5.2.2-chang2022boundedness}
%%-----------
% \sup_{B\subset \mathbb{G}} \dfrac{1}{|B|^{\beta/Q}} \dfrac{\Big\| \big(b - \mathcal{M}_{B} (b) \big) \dchi_{B} \Big\|_{L^{s(\cdot)}(\mathbb{G}) }}{\|\dchi_{B}\|_{L^{s(\cdot)}(\mathbb{G}) }}  < \infty.
%%-----------------
%\end{align}
%%-----------
%Compared with \labelcref{inequ:thm-5.2.2-chang2022boundedness},  \labelcref{inequ:thm-nonlinear-frac-max-lip-4} gives a new characterization for nonnegative Lipschitz functions.
%-------
\item  For the case of $p(\cdot)$ and $q(\cdot)$  being constants, the result above was proved in \cite{wu2023some}, and the equivalence of \labelcref{enumerate:thm-nonlinear-frac-max-lip-1,enumerate:thm-nonlinear-frac-max-lip-2,enumerate:thm-nonlinear-frac-max-lip-4}  was proved in  \cite{wu2023characterizationlip} (for $\alpha=0$).
%-------
\item For the Euclidean case, the result above can be found in \cite{zhang2019some}.
%-------
\end{enumerate}
%--------------------------
\end{remark}
%-----------

So, we give the following corollary from  \cref{thm:nonlinear-frac-max-lip-variable}.

%-----------------
\begin{corollary}  \label{cor:nonlinear-frac-max-lip-variable}
%---------------
  Let  $0 <\beta <1$   and $b\in L_{\loc}^{1}(\mathbb{G})$. %$b$ be a locally integrable function on $\mathbb{G}$.
 Then the following statements are equivalent:
%--------------
\begin{enumerate}[label=(C.\arabic*)]  %,itemindent=2em
%%--------------
\item  $b\in  \Lambda_{\beta}(\mathbb{G})$ and $b\ge 0$.
 %-----------
    \label{enumerate:cor-nonlinear-frac-max-lip-variable-1}
%------------------
   \item  The commutator $ [b,\mathcal{M}] $ is bounded from $L^{p(\cdot)}(\mathbb{G})$ to $L^{q(\cdot)}(\mathbb{G})$ for all $(p(\cdot), q(\cdot))\in   \mathscr{B}^{\beta}(\mathbb{G}) $.
%-----------
    \label{enumerate:cor-nonlinear-frac-max-lip-variable-2}
%--------------------------------
\item  The commutator $ [b,\mathcal{M}] $ is bounded from $L^{p(\cdot)}(\mathbb{G})$ to $L^{q(\cdot)}(\mathbb{G})$   for some  $(p(\cdot), q(\cdot))\in   \mathscr{B}^{\beta}(\mathbb{G}) $.
%-----------
    \label{enumerate:cor-nonlinear-frac-max-lip-variable-3}
%------------
   \item There exists an  $s(\cdot) \in   \mathscr{B} (\mathbb{G}) $   such that
%-----------
\begin{align} \label{inequ:cor-nonlinear-frac-max-lip-variable-4}
%-----------
\sup_{B\subset \mathbb{G}} \dfrac{1}{|B|^{\beta/Q}} \dfrac{\Big\| \big(b -\mathcal{M}_{B} (b) \big) \dchi_{B} \Big\|_{L^{s(\cdot)}(\mathbb{G}) }}{\|\dchi_{B}\|_{L^{s(\cdot)}(\mathbb{G}) }}  < \infty.
%-----------------
\end{align}
%-----------
%-----------
    \label{enumerate:cor-nonlinear-frac-max-lip-variable-4}
%------------
   \item   \labelcref{inequ:cor-nonlinear-frac-max-lip-variable-4} holds for all  $s(\cdot) \in   \mathscr{B} (\mathbb{G}) $.
   %For all    $s(\cdot) \in   \mathscr{B} (\mathbb{G}) $, we have    \labelcref{inequ:cor-nonlinear-frac-max-lip-variable-4} holds.
%-----------
\label{enumerate:cor-nonlinear-frac-max-lip-variable-5}
%---------
\end{enumerate}
%-------------
%----------
\end{corollary}
%------------

%--------------------
\begin{remark}   \label{rem.cor-nonlinear-frac-max-lip-variable}
%-------------
   \cref{cor:nonlinear-frac-max-lip-variable} improves the result of \cite[Theorem 5.2.2]{chang2022boundedness} essentially.
%------------
\end{remark}
%----------

Next,  we consider the mapping properties of
%Next, we consider some necessary and sufficient conditions for the boundedness of the maximal commutator
$\mathcal{M}_{\alpha,b}$ on variable Lebesgue spaces in the context of  stratified Lie group when $b$ belongs to a Lipschitz space.

%-----------------
\begin{theorem} \label{thm:frac-max-lip-variable}
%---------------
 Let  $0 <\beta <1$, $0< \alpha<\alpha+\beta<Q$  and  $b\in L_{\loc}^{1}(\mathbb{G})$.
 Then the following assertions are equivalent:
%--------------
\begin{enumerate}[label=(T.\arabic*)]   %\arabic,align=left
%%--------------
\item   $b\in  \Lambda_{\beta}(\mathbb{G})$.
 %-----------
    \label{enumerate:thm-frac-max-lip-1}
%------------------
   \item The commutator $\mathcal{M}_{\alpha,b} $ is bounded from $L^{p(\cdot)}(\mathbb{G})$ to $L^{q(\cdot)}(\mathbb{G})$ for all $(p(\cdot), q(\cdot))\in   \mathscr{B}^{\alpha+\beta}(\mathbb{G}) $.
%-----------
    \label{enumerate:thm-frac-max-lip-2}
%--------------------------------
\item  The commutator $\mathcal{M}_{\alpha,b}$ is bounded from $L^{p(\cdot)}(\mathbb{G})$ to $L^{q(\cdot)}(\mathbb{G})$   for some  $(p(\cdot), q(\cdot))\in   \mathscr{B}^{\alpha+\beta}(\mathbb{G}) $.
%-----------
    \label{enumerate:thm-frac-max-lip-3}
%------------
   \item  For some  $s(\cdot) \in   \mathscr{B} (\mathbb{G}) $   such that % There exists an  $s(\cdot) \in   \mathscr{B} (\mathbb{G}) $   such that
%-----------
\begin{align} \label{inequ:thm-frac-max-lip-4}
%-----------
\sup_{B\subset \mathbb{G}} \dfrac{1}{|B|^{\beta/Q}} \dfrac{\Big\| \big(b -b_{B}   \big) \dchi_{B} \Big\|_{L^{s(\cdot)}(\mathbb{G}) }}{\|\dchi_{B}\|_{L^{s(\cdot)}(\mathbb{G}) }}  < \infty.
%-----------------
\end{align}
%-----------
%-----------
    \label{enumerate:thm-frac-max-lip-4}
%------------
   \item  For all    $s(\cdot) \in   \mathscr{B} (\mathbb{G}) $, we have    \labelcref{inequ:thm-frac-max-lip-4} holds.
%-----------
\label{enumerate:thm-frac-max-lip-5}
%-----------------
\end{enumerate}
%-------------
%----------
\end{theorem}
%--------------

When $\alpha=0$, we get the following result from  \cref{thm:frac-max-lip-variable}.
%-----------------
\begin{corollary}  \label{cor:frac-max-lip-variable}
%---------------
  Let  $0 <\beta <1$   and  $b\in L_{\loc}^{1}(\mathbb{G})$.
 Then the following statements are equivalent:
%--------------
\begin{enumerate}[label=(C.\arabic*)]   %,itemindent=2em
%%--------------
\item  $b\in  \Lambda_{\beta}(\mathbb{G})$.
 %-----------
    \label{enumerate:cor-frac-max-lip-variable-1}
%------------------
   \item The commutator $\mathcal{M}_{b} $ is bounded from $L^{p(\cdot)}(\mathbb{G})$ to $L^{q(\cdot)}(\mathbb{G})$ for all $(p(\cdot), q(\cdot))\in   \mathscr{B}^{\beta}(\mathbb{G}) $.
%-----------
    \label{enumerate:cor-frac-max-lip-variable-2}
%--------------------------------
\item  The commutator $\mathcal{M}_{b}$ is bounded from $L^{p(\cdot)}(\mathbb{G})$ to $L^{q(\cdot)}(\mathbb{G})$   for some  $(p(\cdot), q(\cdot))\in   \mathscr{B}^{\beta}(\mathbb{G}) $.
%-----------
    \label{enumerate:cor-frac-max-lip-variable-3}
%------------
   \item For some  $s(\cdot) \in   \mathscr{B} (\mathbb{G}) $   such that % There exists an  $s(\cdot) \in   \mathscr{B} (\mathbb{G}) $   such that
%-----------
\begin{align} \label{inequ:cor-frac-max-lip-variable-4}
%-----------
\sup_{B\subset \mathbb{G}} \dfrac{1}{|B|^{\beta/Q}} \dfrac{\Big\| \big(b -b_{B}   \big) \dchi_{B} \Big\|_{L^{s(\cdot)}(\mathbb{G}) }}{\|\dchi_{B}\|_{L^{s(\cdot)}(\mathbb{G}) }}  < \infty.
%-----------------
\end{align}
%-----------
%-----------
    \label{enumerate:cor-frac-max-lip-variable-4}
%------------
   \item  For all    $s(\cdot) \in   \mathscr{B} (\mathbb{G}) $, we have    \labelcref{inequ:cor-frac-max-lip-variable-4} holds.
%-----------
\label{enumerate:cor-frac-max-lip-variable-5}
%---------
\end{enumerate}
%-------------
%----------
\end{corollary}
%------------

%--------------
\begin{remark}   \label{rem.frac-max-lip-variable}
%-----------
\begin{enumerate}[ label=(\roman*)]  %,itemindent=-0.3em,itemindent=1.5em
%-------
\item In \cref{cor:frac-max-lip-variable},   the    equivalence of
\labelcref{enumerate:cor-frac-max-lip-variable-1,enumerate:cor-frac-max-lip-variable-2}  was proved in \cite{chang2022boundedness} (see Theorem 5.2.1).
%-------
\item  For the case of $p(\cdot)$ and $q(\cdot)$  being constants, the result above was proved in \cite{wu2023some} (for $0<\alpha<Q$), and the equivalence of \labelcref{enumerate:cor-frac-max-lip-variable-1,enumerate:cor-frac-max-lip-variable-2,enumerate:cor-frac-max-lip-variable-3}  was proved in  \cite{wu2023characterizationlip} (for $\alpha=0$).
%-------
\item  In the Euclidean case, the results above can be found in \cite{zhang2019some} (for $0<\alpha<Q$) and \cite{zhang2019characterization} (for $ \alpha=0$) separately.
%-------
\end{enumerate}
%--------------------------
\end{remark}
%-----------

 Throughout this paper, the letter $C$  always stands for a constant  independent of the main parameters involved and whose value may differ from line to line.
In addition, we  give some notations. Here and hereafter $|E|$  will always denote the Haar measure of a measurable set $E$ of $\mathbb{G}$ and by  \raisebox{2pt}{$\dchi_{E}$} denotes the  characteristic function of a measurable set $E \subset\mathbb{G}$.
 And for any $f\in L_{\loc}^{1}(\mathbb{G})$, set $f_{E}=\frac{1}{|E|} \int_{E} f(x) \mathd x$.
%Let $L^{p} ~(1\le p\le \infty)$  be the standard $L^{p} $-space with respect to the Haar measure $\mathd x$.
%For a measurable set $E \subset\mathbb{G}$ and a positive integer $m$, we will use the notation $(E)^{m}=\underbrace{E\times \cdots \times E}_{m}$ sometimes.
%And we will occasionally use the notational $\vec{f}=(f_{1},\dots , f_{m})$, $T(\vec{f})=T(f_{1},\dots , f_{m})$, $\mathd\vec{y}=dy_{1}\cdots  dy_{m}$ and $(x,\vec{y})=(x,y_{1},\dots , y_{m})$ for convenience.

%This paper is organized as follows. In the next section, we recall some basic definitions and known results. In  \cref{sec:proof-mab}, we will prove  \cref{thm:lipschitz-frac-main-1}.   \cref{sec:proof-nonlinear}  is devoted to proving \cref{thm:lipschitz-nonlinear-frac-main-1}.

\section{Preliminaries and lemmas}
\label{sec:preliminary}

To prove the principal results, we first recall some necessary notions and remarks.
Below  we  give some preliminaries concerning stratified Lie groups (or so-called
Carnot groups). We refer the reader to  \cite{folland1982hardy,bonfiglioli2007stratified,stein1993harmonic}.

\subsection{Lie group $\mathbb{G}$}

%For the convenience of the reader we collect in this section the most basic facts

%------------------------
 \begin{definition}%[\citet{krantz1982lipschitz}]
 \label{def:stratified-Lie-algebra-krantz1982lipschitz}
%----------------------------
  Let $m\in \mathbb{Z}^{+}$, $\mathcal{G}$  be a finite-dimensional Lie algebra,  $[X, Y] = XY - YX \in \mathcal{G}$ be Lie bracket with $X,Y  \in \mathcal{G}$.
%-------------------------------
\begin{enumerate}[label=(\arabic*)]  %leftmargin=2em,,,itemindent=1.0em
%--------------------------------------------------------------
\item If $Z \in \mathcal{G}$ is an $m^{\text{th}}$  order Lie bracket and $W \in \mathcal{G}$, then $[Z,W]$ is an $(m + 1)^{\text{st}}$ order  Lie bracket.
%----------------
\item  We call $\mathcal{G}$   a nilpotent Lie algebra of step  $m$ if  $m$ is the smallest integer for which all Lie brackets of order $m+1$ are zero.
%-------
\item   We say that  a   Lie algebra $\mathcal{G}$   is  stratified  if there is a direct sum vector space decomposition
%-----------------
\begin{align}\label{equ:lie-algebra-decomposition}
%-------
 \mathcal{G} =\oplus_{j=1}^{m} V_{j}  = V_{1} \oplus  \cdots \oplus  V_{m}
%----------
\end{align}
%----------
such that $\mathcal{G}$ is $m$-step nilpotent, namely,
%-----------------
\begin{align*}
%-------
 [V_{1},V_{j}] =
%------------
 \begin{cases}
 %----------
 V_{j+1} &  1\le j \le  m-1  \\
 0 & j\ge m
%------------
\end{cases}
%----------
\end{align*}
%------------
 holds.
%------------
\end{enumerate}
%---------------
%----------------
\end{definition}
%----------

It is not difficult to find that the above $V_{1}$  generates the whole of the Lie algebra $\mathcal{G}$ by taking Lie brackets since each element of $ V_{j}~(2\le j \le m)$  is a linear combination of $(j-1)^{\text{th}}$ order Lie bracket  of elements of $ V_{1}$.

With the help of the related  notions of Lie algebra (see \cref{def:stratified-Lie-algebra-krantz1982lipschitz}), the following definition can be obtained.
%------------------------
 \begin{definition}\label{def:stratified-Lie-group}
%----------------------------
  Let $\mathbb{G}$ be a finite-dimensional, connected and simply-connected Lie group   associated with Lie algebra $\mathcal{G}$. Then
%-------------------------------
\begin{enumerate}[label=(\arabic*)]  %leftmargin=2em,,,itemindent=1.0em
%--------------------------------------------------------------
\item  $\mathbb{G}$  is called nilpotent if its Lie algebra $\mathcal{G}$ is nilpotent.
%----------------
\item  $\mathbb{G}$ is said to be stratified if its Lie algebra $\mathcal{G}$ is stratified.
%-------
\item  We call $\mathbb{G}$  homogeneous if it is a nilpotent Lie group whose Lie algebra $\mathcal{G}$ admits a family of dilations $\{\delta_{r}\}$, namely, for $r>0$, $X_{k}\in V_{k}~(k=1,\ldots,m)$,
%-----------------
\begin{align*}
%-------
  \delta_{r} \Big( \sum_{k=1}^{m} X_{k} \Big)  =  \sum_{k=1}^{m} r^{k} X_{k},
%----------
\end{align*}
%------------
which are Lie algebra automorphisms.
%------------
\end{enumerate}
%---------------
%----------------
\end{definition}
%----------

%--------------------------------
\begin{remark}  \label{rem:lie-algebra-decom-zhu2003herz}  %
%--------------------------------
Set $\mathcal{G} =  \mathcal{G}_{1}\supset  \mathcal{G}_{2} \supset \cdots \supset  \mathcal{G}_{m+1} =\{0\}$   denote the lower central series of  $\mathcal{G}$, and $X=\{X_{1},\dots,X_{n}\}$ be a basis for $V_{1}$ of $\mathcal{G}$.
%-------------------------------
\begin{enumerate}[label=(\roman*) ]  %leftmargin=2em,,itemindent=-0.3em
%--------------------------------------------------------------
\item  (see \cite{zhu2003herz})  The direct sum decomposition  \labelcref{equ:lie-algebra-decomposition} can be constructed by identifying each $\mathcal{G}_{j}$ as a vector subspace of $\mathcal{G}$ and setting $ V_{m}=\mathcal{G}_{m}$ and $ V_{j}=\mathcal{G}_{j}\setminus \mathcal{G}_{j+1}$ for $j=1,\ldots,m-1$.
%------------
\item   (see \cite{folland1979lipschitz}) The number $Q=\trace A =\sum\limits_{j=1}^{m} j\dim(V_{j})$ is called the homogeneous dimension of $\mathcal{G}$, where $A$ is a diagonalizable linear transformation of  $\mathcal{G}$ with positive eigenvalues.
%------------
\item  (see \cite{zhu2003herz} or \cite{folland1979lipschitz})   %The  homogeneous dimension of  $\mathbb{G}$ at infinity as the integer $Q$ is given by
The number  $Q$ is also called the homogeneous dimension of $\mathbb{G}$  since $\mathd(\delta_{r}x)=r^{Q}\mathd x$ for all $r>0$, and
%-----------------
\begin{align*}
%-------
 Q = \sum_{j=1}^{m} j \dim(V_{j}) = \sum_{j=1}^{m} \dim(\mathcal{G}_{j}).
%----------
\end{align*}
%------------
\end{enumerate}
%-----------------------
\end{remark}
%------------------------

%\textcolor{red}{
%By the Baker-Campbell-Hausdorff formula for sufficiently small elements $X$ and $Y$ of $\mathcal{G}$ one has
%%------------------
%\begin{align*}
%%-------
% \exp (X) \exp (Y)=  \exp (H(X,Y))\ \text{and}\  H(X,Y)= X+Y +\frac{1}{2}[X,Y]+\cdots,
%%----------
%\end{align*}
%%------------
%where $\exp : \mathcal{G} \to \mathbb{G}$ is the exponential map, $H(X, Y )$ is an infinite linear
%combination of $X$ and $Y$ and their Lie brackets, and the dots denote terms of order higher than two. And the above equation is finite in the case  of  $\mathcal{G}$ is a nilpotent Lie algebra (see also Theorem 1.3.2 in \cite{fischer2016quantization}).
%}

The following properties can be found in \cite{ruzhansky2019hardy}(see Proposition 1.1.1, %\textcolor{red}{or  Proposition 2.1 in \cite{yessirkegenov2019function} }
or Proposition 1.2 in \cite{folland1982hardy}).

%--------------------------------
\begin{proposition}[Exponential mapping and Haar measure] \label{pro:2.1-yessirkegenov2019}
%------------------------------
 Let $\mathcal{G}$ be a nilpotent Lie algebra, and let $\mathbb{G}$ be the corresponding connected and simply-connected nilpotent Lie group. Then we have
%------------------------------
\begin{enumerate}[leftmargin=2em,label=(\arabic*),itemindent=1.0em]  %,itemindent=-0.3em
%--------------------------------
\item   The exponential map  $\exp: \mathcal{G} \to \mathbb{G}$  is a diffeomorphism. Furthermore, the group law $(x,y) \mapsto xy$ is a polynomial map if  $\mathbb{G}$ is identified with $\mathcal{G}$ via $\exp$.
%-------------------------------------
\item  If $\lambda$ is a Lebesgue measure on  $\mathcal{G}$, then $\exp\lambda$ is a bi-invariant Haar measure on  $\mathbb{G}$ (or a bi-invariant Haar measure $\mathd  x$ on  $\mathbb{G}$  is just the lift of Lebesgue measure on  $\mathcal{G}$ via $\exp$).
%-------------------------------
\end{enumerate}
%-----------------------------
\end{proposition}
%------------------------

Thereafter, we use $Q$ to denote the homogeneous dimension of  $\mathbb{G}$, $y^{-1}$ represents the inverse of $y\in \mathbb{G}$,  $y^{-1}x$ stands for the group multiplication of $y^{-1}$  by $x$ and the group identity  element of $\mathbb{G}$ will be referred to as the origin denotes by $\theta$.

A homogeneous norm $\rho:~ x\to \rho(x)$  defined on $\mathbb{G}$ is a continuous function  from $\mathbb{G}$ to $[0,\infty)$, which is  $C^{\infty}$ on $\mathbb{G}\setminus\{\theta\}$ and satisfies
%-----------------
\begin{align*}
%-------
\begin{cases}
%------------
 \rho(x^{-1}) =  \rho(x), \\
 \rho(\delta_{t}x) =  t\rho(x) \ \ \text{for all}~  x \in \mathbb{G} ~\text{and}~ t > 0, \\
%------------
 \rho(\theta) =  0.
%------------
\end{cases}
%----------
\end{align*}
%------------

Set
%-----------
\begin{align*}
%-------
  \rho(x,y) := \rho(x^{-1}y)=\rho(y^{-1}x), \ \  \ \  \forall~ x,y\in \mathbb{G}.
%----------
\end{align*}
%-----------
It defines a quasi-distance in the sense of Coifman-Weiss,
namely, there exists a constant $c_{0} \ge 1$ such that the pseudo-triangle inequality
%-----------
\begin{align*}
%-------
  \rho(x,y) \le c_{0}(\rho(x,z) + \rho(z,y)) %, \ \  \ \  \forall~ x,y,z\in \mathbb{G}.
%----------
\end{align*}
%-----------
holds for every $x,y,z\in \mathbb{G}$.
In particular, the Carnot–Carath\'{e}odory distance $d_{cc}$ generated by a vector field $\{X_{j}\}_{1\le j\le n}$ \cite{bonfiglioli2007stratified}, which is not smooth on stratified Lie group $\mathbb{G}\setminus\{\theta\}$, is equivalent to  %an arbitrary homogeneous symmetric  metric $\rho(\cdot)$,
the pseudo-metric $\rho(\cdot,\cdot)$, that is, there exist $C_{1}, C_{2}>0$ satisfying
%-----------
\begin{align*}
%-------
  C_{1}\rho(x,y) \le   d_{cc}(x,y)  \le  C_{2} \rho(x,y)),   \ \  \ \  \forall~ x,y\in \mathbb{G}.
%----------
\end{align*}
%-----------
%In general, the Carnot–Carath\'{e}odory distance $d_{cc}$ generated by a vector field $\{X_{j}\}_{1\le j\le n}$  on a stratified Lie group is not smooth over  $\mathbb{G}\setminus\{\theta\}$.

 With the norm above, we define the $\mathbb{G}$ ball centered at $x$ with radius $r$ by $B(x, r) = \{y \in \mathbb{G} : \rho(y^{-1}x) < r\}$,  and by $\lambda B$ denote the ball $B(x,\lambda r)$  with $\lambda>0$, let $B_{r} = B(\theta, r) = \{y \in \mathbb{G}  : \rho(y) < r\}$ be the open ball centered at $\theta$ with radius $r$,  which is the image under $\delta_{r}$ of $B(\theta, 1)$.
 And by $\sideset{^{\complement}}{}  {\mathop {B(x,r)}} = \mathbb{G}\setminus B(x,r)= \{y \in \mathbb{G} : \rho(y^{-1}x) \ge r\}$ denote the complement of $B(x, r)$.  Let  $|B(x,r)|$ be the Haar measure of the ball  $B(x,r)\subset \mathbb{G}$, and
 there exists $c_{1} =c_{1} (\mathbb{G})$ such that
%----------------
\begin{align*}
%-------
  |B(x,r)| = c_{1} r^{Q}, \ \  \ \   x\in \mathbb{G}, r>0.
%----------
\end{align*}
%-----------
In addition,  the Haar measure of a homogeneous Lie group  $\mathbb{G}$  satisfies the doubling condition (see pages 140 and 501,\cite{fischer2016quantization}), i.e. $\forall ~ x\in \mathbb{G}$, $r>0$, $\exists~ C$, such that
%----------------
\begin{align*}
%-------
   |B(x,2r)| \le C |B(x,r)|.
%----------
\end{align*}

The most basic partial differential operator in a stratified Lie group is the sub-Laplacian associated with $X=\{X_{1},\dots,X_{n}\}$, i.e., the second-order partial differential operator on
 $\mathbb{G}$  given by
%--------------------------------
\begin{align*}
%-------
 \mathfrak{L} =  \sum_{i=1}^{n} X_{i}^{2}.
%----------
\end{align*}
%-----------

%Roughly speaking, there are three versions of Hardy type inequalities on stratified groups available in the literature:
%%-------------------------------
%\begin{enumerate}[label=(\arabic*)]  %leftmargin=2em,,,itemindent=1.0em
%%--------------------------------------------------------------
%\item Using the homogeneous quasi-norm, sometimes called the $\mathfrak{L}$-gauge, given by the appropriate power of the fundamental solution of the sub-Laplacian $\mathfrak{L}$. Thus, if $d(x)$ is the L-gauge, then $d(x)^{2−Q}$ is a constant multiple of Folland's \cite{folland1975subelliptic}  fundamental solution of the sub-Laplacian $\mathfrak{L}$, with $Q$ being the homogeneous dimension of the stratified group $\mathbb{G}$; these will be discussed in Chapter 7 in \cite{ruzhansky2019hardy}.
%%----------------
%\item   Using the Carnot–Carath\'{e}odory distance, i.e., the control distance associated to the sub-Laplacian.
%%-------
%\item   Using the Euclidean distance on the first stratum of the group.
%%------------
%\end{enumerate}
%%---------------

\subsection{Lipschitz spaces on $\mathbb{G}$}

Next we give the definition of the Lipschitz spaces on $\mathbb{G}$, and state some basic properties and useful lemmas.

%--------------------------------
\begin{definition}[Lipschitz-type spaces on $\mathbb{G}$]   \label{def.lip-space} \
%--------------------------------排列环境--------------------------------
\begin{enumerate}[ label=(\arabic*)]
%-------------------------------------
\item   Let $0<\beta <1$, we say a function $b$ belongs to the Lipschitz space $\Lambda_{\beta}(\mathbb{G}) $ if there exists a constant $C>0$ such that for all  $x,y\in \mathbb{G}$,
%----------------
\begin{align*}
%-------
  |b(x)-b(y)|   &\le C(\rho(y^{-1}x))^{\beta},
%----------
\end{align*}
%------------
where $\rho$ is the homogenous norm. The smallest such constant $C$ is called the $\Lambda_{\beta}$  norm of $b$ and is denoted by $\|b\|_{\Lambda_{\beta}(\mathbb{G})}$.
%-----------
    \label{enumerate:def-lip-1}
%-------------------------------------
\item (see \cite{macias1979lipschitz} ) Let $0<\beta <1$ and $1\le p<\infty$.   A locally integrable function $b$ is said to belong to  the space $\lip_{\beta,p}(\mathbb{G}) $ if
 %    The space $\lip_{\beta,p}(\mathbb{G}) $ is defined to be the set of all locally integrable  functions $b$, i.e.,
 there exists a positive constant $C $, such that
%----------------
\begin{align*}
%-------
      \sup_{B\ni x} \dfrac{1}{ |B|^{\beta/Q}}\Big( \dfrac{1}{|B|}  \dint_{B} |b(x)- b_{B}|^{p}\mathd x \Big)^{1/p} \le C,
%----------
\end{align*}
%------------
where the supremum is taken over every ball $B\subset \mathbb{G}$ containing $x$ and $b_{B}=\frac{1}{|B|} \int_{B} b(x) \mathd x$. The least constant $C$   satisfying the conditions above shall   be denoted by $\|b\|_{\lip_{\beta,p}(\mathbb{G})}$.
%-----------
    \label{enumerate:def-lip-2}
%-------------------------------------
%\item (see \cite{macias1979lipschitz}) Let $0<\beta <1$. When  $  p=\infty$, we shall say that a locally integrable  functions $b$   belongs to  $\lip_{\beta,\infty}(\mathbb{G}) $  if there exists a constant $C$ such that
% %----------------
%\begin{align*}
%%-------
%      \esssup_{x\in B} \dfrac{|b(x)- b_{B}|}{ |B|^{\beta/Q}}   \le C
%%----------
%\end{align*}
%%------------
%holds  for every ball $B\subset \mathbb{G}$  with $b_{B}=\frac{1}{|B|} \int_{B} b(x) \mathd x$. And $\|b\|_{\lip_{\beta,\infty}(\mathbb{G})}$ stand for the least constant   $C$  satisfying the conditions above.
%%-----------
%    \label{enumerate:def-lip-3}
%-------------------
\end{enumerate}
%-------------------
\end{definition}
%--------------------

\begin{remark}  \label{rem.Lipschitz-def}
%-------------------------------
\begin{enumerate}[label=(\roman*)]
%--------------------------------------------------------------
\item  In addition to the form of  \cref{def.lip-space} \labelcref{enumerate:def-lip-1}, %Similar to the definition of Lipschitz space $\Lambda_{\beta}(\mathbb{G}) $ in \labelcref{enumerate:def-lip-1},
    we also have the definition form as following  (see  \cite{krantz1982lipschitz,chen2010lipschitz,fan1995characterization} et al.)
%----------------
\begin{align*}
%-------
 \|b\|_{\Lambda_{\beta}(\mathbb{G})}&= %\sup_{x,y\in \mathbb{G}\atop y\neq e} \dfrac{|b(xy)- b(x)|}{(\rho(y))^{\beta}}   =
 \sup_{x,y\in \mathbb{G} \atop x\neq y} \dfrac{|b(x)-b(y)|}{(\rho(y^{-1}x))^{\beta}}<\infty.
%----------
\end{align*}
%------------
And $\|b\|_{\Lambda_{\beta}(\mathbb{G})} =0$   if and only if $b$ is constant.   %第一个=：P41 \cite{krantz1982lipschitz} 第二个=： Def 2 in \cite{chen2010lipschitz};  def 1 in \cite{fan1995characterization}
%-------------------------------------
\item  In \labelcref{enumerate:def-lip-2},  when   $p=1$, we have
%----------------
\begin{align*}
%-------
     \|b\|_{\lip_{\beta,1}(\mathbb{G})} =\sup_{B\ni x} \dfrac{1}{ |B|^{\beta/Q}}\Big( \dfrac{1}{|B|}  \dint_{B} |b(x)- b_{B}| \mathd x \Big) :=\|b\|_{\lip_{\beta}(\mathbb{G})}.
%----------
\end{align*}
%------------
%\item    There are two basically different approaches to Lipschitz classes on the $n$-dimensional  Euclidean space. Lipschitz classes can be defined via Poisson (or Weierstrass) integrals  of $L^{p}$-functions, or, equivalently, by  means of higher order difference operators (see \cite{meda1988lipschitz}).
%------------
\end{enumerate}
%--------------------------
\end{remark}
%------------------------

Since stratified Lie groups can be  regarded as  a special case of spaces of homogeneous type in the sense of Coifman-Weiss, hence, the following results can be  derived from  \cite{macias1979lipschitz,chen2010lipschitz,li2003lipschitz}.
%----------------
\begin{lemma} \label{lem:2.2-li2003lipschitz}   %(see \cite{macias1979lipschitz,chen2010lipschitz,li2003lipschitz} )
% Let   $0<\beta<1$ and the function $b(x)$ integrable on bounded subsets of $\mathbb{G}$.
Assume that  $0<\beta<1$ and $b$  is a locally integrable function on $\mathbb{G}$.
%-------------------------------
\begin{enumerate}[label=(\roman*)]
%--------------------------
\item  When $1\le p<\infty$,  then
%----------------
\begin{align*}
%-------
 \|b\|_{\Lambda_{\beta}(\mathbb{G})} &=  \|b\|_{\lip_{\beta}(\mathbb{G})} \approx  \|b\|_{\lip_{\beta,p}(\mathbb{G})}.
%----------
\end{align*}
%------------
\label{enumerate:property-lip-Lie-1}    % Thm 2.2 in  \cite{li2003lipschitz}
%---------
\item   Let balls $B_{1}\subset B_{2}\subset \mathbb{G}$ and $b\in \lip_{\beta,p}(\mathbb{G})$ with $p\in [1,\infty)$. Then there exists a constant $C$ depends on $B_{1}$ and $B_{2}$ only, such that
%----------------
\begin{align*}
%-------
     |b_{B_{1}}- b_{B_{2}} |   &\le    C  \|b\|_{\lip_{\beta,p}(\mathbb{G})} |B_{2}|^{\beta/Q}.
%----------
\end{align*}
%------------
\label{enumerate:property-lip-Lie-2}  % lem 4,5 in \cite{chen2010lipschit}
%---------
\item   When $1\le p<\infty$, then there exists a constant $C$ depends on $\beta$ and $p$ only, such that
%----------------
\begin{align*}
%-------
     | b(x)-  b(y) |   &\le   C  \|b\|_{\lip_{\beta,p}(\mathbb{G})} |B|^{\beta/Q}
%----------
\end{align*}
%------------
holds for any ball $B$ containing $x$ and $y$.
%------------
\label{enumerate:property-lip-Lie-3}
%---------
\end{enumerate}
%---------------------
%----------
\end{lemma}
%------------

%-----------
\subsection{Variable exponent function spaces  on $\mathbb{G}$}
%---------

We now introduce the notion of  variable exponent Lebesgue spaces on stratified
groups and give some properties needed in
the sequel (see \cite{liu2022characterisation} or \cite{liu2019multilinear} for more details). %the respective proofs

We say that a measurable function $p(\cdot)$ is a variable exponent if $p(\cdot): \mathbb{G}\to (0,\infty)$.

%--------------
\begin{definition} \label{def.variable-exponent-Lie}
%-----------------
  Given a measurable function $q(\cdot)$ defined on  a measurable subset $E\subset\mathbb{G}$. Set
$$
p_{-}(E) :=\essinf_{x\in E} p(x),\ \
p_{+}(E) := \esssup_{x\in E} p(x).
$$
For conciseness, we abbreviate $p_{-}(\mathbb{G})$ and $p_{+}(\mathbb{G})$  to $p_{-}$ and $p_{+}$.
%----------
\begin{enumerate}[label=(\arabic*),itemindent=1em] %fullwidth,
%------------ see P14 \cite{cruz2013variable}
\item $p'_{-}=\essinf\limits_{x\in \mathbb{G}} p'(x)=\frac{p_{+}}{p_{+}-1},\ \ p'_{+}= \esssup\limits_{x\in \mathbb{G}} p'(x)=\frac{p_{-}}{p_{-}-1}.$
%----------
% \item  Denote by $\mathscr{P}_{0}(\mathbb{G})$ the set of all measurable functions $ p(\cdot): \mathbb{G}\to(0,\infty)$ such that
%$$0< p_{-}\le p(x) \le p_{+}<\infty,\ \ x\in \mathbb{G}.$$
%------------
\item  Denote by $\mathscr{P}_{1}(\mathbb{G})$ the set of all measurable functions $ p(\cdot): \mathbb{G}\to[1,\infty)$ such that
$$1\le p_{-}\le p(x) \le p_{+}<\infty,\ \ x\in \mathbb{G}.$$
%-----------
  \item Denote by $\mathscr{P}(\mathbb{G})$ the set of all measurable functions $ q(\cdot): \mathbb{G}\to(1,\infty)$ such that
$$1< p_{-}\le p(x) \le p_{+}<\infty,\ \ x\in \mathbb{G}.$$
%----------
 \item  The set $\mathscr{B}(\mathbb{G})$ consists of all  measurable functions  $p(\cdot)\in\mathscr{P}(\mathbb{G})$ satisfying that the Hardy-Littlewood maximal operator $\mathcal{M}$ is bounded on $L^{p(\cdot)}(\mathbb{G})$.
%-------
\end{enumerate}
%--------------
\end{definition}
%------------

%%--------------------
%\begin{remark}   \label{rem.variable-max-bounded}
%%-------------
%If  $ p(\cdot)\in \mathscr{B}(\mathbb{G})$ and $\lambda>1$, then $ \lambda p(\cdot)\in \mathscr{B}(\mathbb{G})$ (the idea is taken from \cite{cruz2006theboundedness} and for the Euclidean case  to see  Remark 2.13 in \cite{cruz2006theboundedness}).
%%------------
%\end{remark}
%%----------

%--------------
\begin{definition}[Variable exponent Lebesgue spaces on $\mathbb{G}$] \label{def.variable-lebesgue-space-Lie}
%-----------------
 Let   $p(\cdot) \in \mathscr{P}(\mathbb{G})$ be a measurable function.
%----------
 The   variable exponent Lebesgue space $L^{p(\cdot)}(\mathbb{G})$  is defined  as follows
% consists of all measurable functions $f$ satisfying  which satisfy also the following condition
% equipped with the following   norm
%----------------
\begin{align*}
%-------
  L^{p(\cdot)}(\mathbb{G})= \Big \{f~ \text{is measurable function}: \int_{\mathbb{G}} \Big( \frac{|f(x)|}{\eta} \Big)^{p(x)} \mathrm{d}x<\infty ~\text{for some constant}~ \eta>0  \Big \}.
%----------
\end{align*}
%------------
 It is known that the set   $L^{p(\cdot)}(\mathbb{G})$ becomes a Banach  space with respect to the Luxemburg norm
 \begin{equation*}
  \|f\|_{ p(\cdot)}= \|f\|_{L^{p(\cdot)}(\mathbb{G})}=\inf \Big\{ \eta>0: \int_{\mathbb{G}} \Big( \frac{|f(x)|}{\eta} \Big)^{p(x)} \mathrm{d}x \le 1 \Big\}.
\end{equation*}
%----------
\end{definition}
%------------

%--------------
\begin{definition}[$\log$-H\"{o}lder continuity] \label{def.log-holder-liu2022characterisation} %\cite{liu2022characterisation}
%-----------------
 Let $\Omega\subset\mathbb{G}$ and $q(\cdot) \in \mathscr{P}(\Omega)$ be a measurable function.
%----------
\begin{enumerate}[label=(\arabic*),itemindent=1em] %fullwidth,
%------------
\item Denote by  $\mathscr{C}_{0}^{\log}(\Omega)$ the set of all local  $\log$-H\"{o}lder continuous  functions     $q(\cdot)$ which satisfies
 \begin{equation}     \label{inequ:2.2-liu2022characterisation}
 |q(x)-q(y)| \le \frac{C}{\log (e+1/\rho(x,y))}
\end{equation}
for all $x,y\in \Omega$, where $C$ denotes a %universal
positive constant.
%----------
 \item  The set $\mathscr{C}_{\infty}^{\log}(\Omega)$ consists of all $ q(\cdot)$ which  satisfies $\log$-H\"{o}lder decay condition with basepoint  $\theta\in \mathbb{G}$, if there exists $q_{\infty}\in \mathbb{R}$  such that

\begin{equation} \label{inequ:2.3-liu2022characterisation}
 |q(x)-q_{\infty}| \le \frac{C}{\log (e+ \rho(\theta,x))}
\end{equation}
%------------
 for any $x \in\Omega$, where $C$ denotes a  positive constant. % where $q_{\infty}=\lim\limits_{\rho(x)\to\infty} q(x)$.
%-----------
 \item Denote by $\mathscr{C}^{\log}(\Omega) =\mathscr{C}_{0}^{\log}(\Omega) \bigcap \mathscr{C}_{\infty}^{\log}(\Omega)$ the set consists of all global log-H\"{o}lder continuous functions $q(\cdot)$ in $\Omega$.
%-------
% \item    Denote by $\mathscr{P}^{\log}(\mathbb{G}) =\mathscr{C}(\mathbb{G}) \bigcap \mathscr{P}(\mathbb{G})$.
%-------
\end{enumerate}
%--------------
\end{definition}
%------------

The first two  parts of \cref{rem.log-holder-liu2022characterisation}  can be found in \cite{liu2022characterisation}.
%------------------
\begin{remark}  \label{rem.log-holder-liu2022characterisation}
%------------------
\begin{enumerate}[ label=(\roman*)]  %,itemindent=-0.3em,itemindent=1.5em
%-------------------
\item    If $q(\cdot)$ is global log-H\"{o}lder continuous in $\mathbb{G}$, employing  \labelcref{inequ:2.2-liu2022characterisation}, then we have
%-----------------
\begin{align}  \label{inequ:2.4-liu2022characterisation}
%-------
 |q(x)-q(y)| \le -\dfrac{C}{\log (\rho(x,y))} \qquad \text{for}\  \rho(x,y)\le \frac{1}{2},
%----------
\end{align}
%------------
and \labelcref{inequ:2.3-liu2022characterisation} is equivalent to
%-----------------
\begin{align}  \label{inequ:2.5-liu2022characterisation}
%-------
 |q(x)-q(y)| \le  \frac{C}{\log (e+ \rho(\theta,x))} \qquad \text{for}\  \rho(\theta,y)\ge \rho(\theta,x).
%----------
\end{align}
%------------
\item  Note that $q_{\infty}=\lim\limits_{x\to\infty} q(x)$ exists in view of \labelcref{inequ:2.5-liu2022characterisation}.
%------------
\item Especially, we say  $q(\cdot)\in \mathscr{C}^{\log}(\mathbb{G}) $ with
basepoint  $\theta\in \mathbb{G}$ if both conditions \labelcref{inequ:2.4-liu2022characterisation,inequ:2.5-liu2022characterisation}
are satisfied.
%------------
\item In fact, when $q(\cdot)\in\mathscr{P}(\mathbb{G})$, if $q(\cdot)\in \mathscr{C}^{\log}(\mathbb{G}) $, then so do  $q'(\cdot)$ and $1/q(\cdot)$ since
%----------------
\begin{align*}
%-------
  |q'(x)-q'(y)| &\le \frac{|q(x)-q(y)|}{(q_{-}-1)^{2}}
%------------
\\ \intertext{and}
%--------
 \Big| \frac{1}{q(x)}-\frac{1}{q(y)}\Big| &\le \frac{|q(x)-q(y)|}{(q_{-})^{2}}.
%----------
\end{align*}
%---------
%------------
\end{enumerate}
%---------------
\end{remark}
%-----------

\subsection{Some auxiliary lemmas}

The first part of \cref{lem.cor-2.4-variable-max-bounded} may be obtained from  \cite[Corollary 2.4]{liu2019multilinear} (or \cite{adamowicz2015maximal} and \cite{cruz2018boundedness}).
%  \cite{adamowicz2015maximal}(see Corollary 1.8).
  By standard arguments and  elementary inferences, the second of \cref{lem.cor-2.4-variable-max-bounded} can  be derived  as well (see \cref{rem.log-holder-liu2022characterisation} or refer to Diening \cite[Theorem 8.1]{diening2005maximalf}   and   Cruz-Uribe et al. \cite[Theorem 1.2]{cruz2006theboundedness}).

\begin{lemma}  \label{lem.cor-2.4-variable-max-bounded}
Let $p(\cdot)\in \mathscr{P}(\mathbb{G} )$.
%----------
\begin{enumerate}[label=(\arabic*)] %fullwidth,,itemindent=1em
%------------
\item  If $p(\cdot)\in \mathscr{C}^{\log}(\mathbb{G})$,  then  $ p(\cdot)\in \mathscr{B}(\mathbb{G})$.
%----------
 \item  The  following conditions are equivalent: %\color{red}
%------------
  \begin{enumerate}[label=(\roman*),itemindent=1em,align=left]
%-----------
  \item  $ p(\cdot)\in \mathscr{B}(\mathbb{G})$,
 \item   $p'(\cdot)\in \mathscr{B}(\mathbb{G})$,
  \item    $ p(\cdot)/p_{0}\in \mathscr{B}(\mathbb{G})$ for some $1<p_{0}<p_{-}$,
 \item    $ (p(\cdot)/p_{0})'\in \mathscr{B}(\mathbb{G})$ for some $1<p_{0}<p_{-}$,
%-----------
\end{enumerate}
%----------
where $r'$ stands for the conjugate exponent of $r$, viz., $1=\frac{1}{r(\cdot)} + \frac{1}{r'(\cdot)}$.
%----------
\end{enumerate}
%--------------
\end{lemma}
%-------------

%--------------------
\begin{remark}   \label{rem.variable-max-bounded}
%-------------
If  $ p(\cdot)\in \mathscr{B}(\mathbb{G})$ and $\lambda>1$, then $ \lambda p(\cdot)\in \mathscr{B}(\mathbb{G})$ is obvious (the idea is taken from \cite{cruz2006theboundedness} and for the Euclidean case  to see  Remark 2.13 in \cite{cruz2006theboundedness}).
%------------
\end{remark}
%----------

For convenience, we introduce a notation $\mathscr{B}^{\gamma}(\mathbb{G})$ as follows (similar to Definition 1.5 in \cite{zhang2019some}).

%--------------
\begin{definition} \label{def.variable-exponent-pair-Lie}
%------------
 Let $0<\gamma<Q$. We say that an ordered pair of variable exponents $(p(\cdot),q(\cdot)) \in \mathscr{B}^{\gamma}(\mathbb{G})$, if $p(\cdot) \in \mathscr{P}(\mathbb{G})$ with $p_{+}<Q/\gamma$ and $1/q(\cdot) = 1/p(\cdot)-\gamma/Q$ with $q(\cdot)(Q-\gamma)/Q \in \mathscr{B}(\mathbb{G})$.
%-------
\end{definition}
%------------

%--------------------
\begin{remark}   \label{rem.variable-pair-Lie}
%-------------
\begin{enumerate}[label=(\roman*) ]
%---------------
\item  From \cref{lem.cor-2.4-variable-max-bounded} (1),  it is easy to see that the condition $p(\cdot)=q(\cdot)(Q-\gamma)/Q \in \mathscr{B}(\mathbb{G})$ can be replaced by $p(\cdot)\in \mathscr{C}^{\log}(\mathbb{G})$ in \cref{def.variable-exponent-pair-Lie}.
%------------
\item  Moreover,   the condition $q(\cdot)(Q-\gamma)/Q \in \mathscr{B}(\mathbb{G})$ is equivalent to saying that there exists a $q_{0}$ with $1<Q/(Q-\gamma)<q_{0}<\infty$ such that $ q(\cdot)/q_{0}\in \mathscr{B}(\mathbb{G})$ since (similar case of Euclidean space see \cite{cruz2006theboundedness} for  further information)
%--------
\begin{align*}
%-------
  \frac{q(\cdot)}{Q/(Q-\gamma)} =    \frac{q(\cdot)}{q_{0}}  \frac{q_{0}}{Q/(Q-\gamma)}  \ \text{and}\  \frac{q_{0}}{Q/(Q-\gamma)}>1.
%----------
\end{align*}
%------
%------------
\item  Especially, $p(\cdot)=q(\cdot)(Q-\gamma)/Q \in \mathscr{B}(\mathbb{G})$  implies $q(\cdot) \in \mathscr{B}(\mathbb{G})$ in \cref{def.variable-exponent-pair-Lie}.
%--------
\end{enumerate}
%------------

%------------
\end{remark}
%----------

The part \labelcref{enumerate:holder-Lie} in following lemma is known as the  H\"{o}lder's inequality on Lebesgue spaces over Lie groups $\mathbb{G}$, it can also be deduced  from \cite{rao1991theory} or \cite{guliyev2022some}, when  Young function $\Phi(t)=t^{p}$ and its complementary function $\Psi(t)=t^{q}$ with $\frac{1}{p}+\frac{1}{q}=1$.  And   the part \labelcref{enumerate:variable-holder-Lie} can be  founded in \cite[Lemma 2.1]{liu2019multilinear}
%--------------------------------
\begin{lemma}[Generalized H\"{o}lder's inequality on  $\mathbb{G}$] \label{lem:holder-inequality-Lie-group}
Let   $\mathbb{G}$ be a measurable set.
%--------------------------
\begin{enumerate}[label=(\roman*)]
%------------
\item Suppose that  $1\le p,q \le\infty$ with $\frac{1}{p}+\frac{1}{q}=1$,   and measurable functions $f\in L^{p}(\Omega)$ and $g\in L^{q}(\mathbb{G})$.  Then there exists a positive constant $C$ such that
%-------------
\begin{align*}
%-------
   \dint_{\mathbb{G}} |f(x)g(x)|  \mathrm{d}x \le C \|f\|_{L^{p}(\mathbb{G})} \|g\|_{L^{q}(\mathbb{G})}.
%----------
\end{align*}
%------------
%-----------
    \label{enumerate:holder-Lie}
%----------------
\item  Assume that    $ p(\cdot), q(\cdot) \ge 1$ on $\mathbb{G}$ and   $r(\cdot) $ defined by
%-------------
\begin{align*}
%-------
   \dfrac{1}{r(x) }=\dfrac{1}{p(x)}+ \dfrac{1}{q(x)}  \qquad  \text{for} \ \almostevery~  x \in \mathbb{G}.
%----------
\end{align*}
%--------
 Then there exists a positive constant $C$ such that the inequality
%----------------
\begin{align*}
%-------
 \|fg\|_{r(\cdot)}\le C \|f \|_{p(\cdot)}  \|g \|_{q(\cdot)}
%----------
\end{align*}
%------------
holds for all  $f \in L^{p(\cdot)}(\mathbb{G})$ and  $g \in L^{q(\cdot)}(\mathbb{G})$.
%-----------
    \label{enumerate:variable-holder-Lie}
%----------
\item When $r(\cdot)= 1$ in  \labelcref{enumerate:variable-holder-Lie}  as mentioned above, we have $ p(\cdot), q(\cdot)  \ge 1$ and   $\frac{1}{p(\cdot)}+ \frac{1}{q(\cdot)}=1$ almost everywhere.
Then there exists a positive constant $C$ such that the inequality
%----------------
\begin{align*}
%-------
    \dint_{\mathbb{G}} |f(x)g(x)|  \mathrm{d}x  \le C \|f \|_{p(\cdot)}  \|g \|_{q(\cdot)}
%----------
\end{align*}
%------------
holds for all   $f \in L^{p(\cdot)}(\mathbb{G})$ and  $g \in L^{q(\cdot)}(\mathbb{G})$.
%-----------
    \label{enumerate:variable-holder-Lie-1}
%------
\end{enumerate}
%--------------
\end{lemma}
%-------------

The following properties for the characteristic function   are  required as well.
By elementary calculations, the   part   \labelcref{enumerate:charact-norm-Lie} can be obtained from the properties of Lie group (or referring to  \cite{guliyev2022some} when  Young function $\Phi(t)=t^{p}$).
The both \labelcref{enumerate:charact-norm-Lie-variable} and \labelcref{enumerate:charact-norm-Lie-variable-1} can be found in \cite{liu2022characterisation} (see Lemma 3.2 and Lemma 3.3).
In particular, by standard arguments, the part   \labelcref{enumerate:charact-norm-fraction-Lie-variable} may be deduced from \cref{lem:holder-inequality-Lie-group} \labelcref{enumerate:variable-holder-Lie}.
  Moreover, according to \cref{lem.cor-2.4-variable-max-bounded} and \cref{lem:norm-characteristic-Lie-liu2022characterisation}  \labelcref{enumerate:charact-norm-Lie-variable-1}, the   part \labelcref{enumerate:charact-norm-Lie-variable-dual}  can also be inferred by simple calculations. So, we omit the proofs.
%--------------------------------
\begin{lemma}[Norms of characteristic functions] \label{lem:norm-characteristic-Lie-liu2022characterisation} % 3.2-3.3
%------------------
\begin{enumerate}[ label=(\arabic*)]  %,itemindent=-0.3em,itemindent=1.5em
%---------
\item  Let $0<p<\infty$ and $\Omega\subset  \mathbb{G}$ be a measurable set with finite Haar measure. Then
%------
\begin{align*}
%-------
 \|\dchi_{\Omega}\|_{L^{p}(\mathbb{G})} = \|\dchi_{\Omega}\|_{WL^{p}(\mathbb{G})}  = |\Omega|^{1/p}.
%----------
\end{align*}
%---------
    \label{enumerate:charact-norm-Lie}
%-------
\item Let $p(\cdot) \in \mathscr{P}(\mathbb{G})$. Then the followings  statements  are equivalent:
%------------------
\begin{enumerate}[ label=(\roman*)]  %,itemindent=-0.3em,itemindent=1.5em
%-------------------
\item    $ p(\cdot)\in \mathscr{C}_{0}^{\log}(\mathbb{G})$.
%------------
\item  For a given ball $B\subset \mathbb{G}$ and  all $x \in B$, there exists a positive constant C such that
%---------
\begin{align*}
%-------
  |B|^{p(x)-p_{+}(B)}  &\le C.
%----------
\end{align*}
%---------
\item  For a given ball $B\subset \mathbb{G}$ and  all $x \in B$, there exists a positive constant C such that
%---------
\begin{align*}
%-------
  |B|^{p_{-}(B)-p(x)}  &\le C.
%----------
\end{align*}
%---------
\end{enumerate}
%---------
    \label{enumerate:charact-norm-Lie-variable}
%-------
\item Suppose  $ p(\cdot)\in \mathscr{P}(\mathbb{G})\bigcap\mathscr{C}^{\log}(\mathbb{G})$.  For a given ball  $B=B(x,r)\subset \mathbb{G}$  and  all $x \in B$.
%------
\begin{enumerate}[ label=(\roman*)]  %,itemindent=-0.3em,itemindent=1.5em
%-------------------
\item  If   $|B|\le 1$,  then
%---------
\begin{align*}
%-------
  \|\dchi_{B}\|_{p(\cdot)}\sim |B|^{\frac{1}{p(x)}} \sim |B|^{\frac{1}{p_{-}(B)}} \sim |B|^{\frac{1}{p_{+}(B)}}.
%----------
\end{align*}
%---------
\item  If   $|B|\ge 1$,  then
%---------
\begin{align*}
%-------
  \|\dchi_{B}\|_{p(\cdot)}\sim   |B|^{\frac{1}{p_{\infty} }}.
%----------
\end{align*}
%---------
\end{enumerate}
%---------
    \label{enumerate:charact-norm-Lie-variable-1}
%-------
\item Let $0 <\alpha<n$. If  $ p(\cdot) $, $q(\cdot)\in \mathscr{P}(\mathbb{G})$  with   $p_{+}<\frac{Q}{\alpha}$ and $1/q(\cdot) = 1/p(\cdot) - \alpha/Q$, then  there exists a  constant $C>0$ such that
%----------------
\begin{align*}
%-------
  \|\dchi_{B}  \|_{L^{p(\cdot)}(\mathbb{G})}\le C  |B|^{\beta/Q} \|\dchi_{B}  \|_{L^{q(\cdot)}(\mathbb{G})}
%----------
\end{align*}
%------------
holds for all    balls $B \subset \mathbb{G}$.
%-----------
    \label{enumerate:charact-norm-fraction-Lie-variable}
%-------
\item Let $ p(\cdot)\in \mathscr{B}(\mathbb{G})$, then there exists a constant $C>0$ such that
%----------
\begin{align*}
%-------
  \dfrac{1}{|B|}  \|\dchi_{B}\|_{L^{p(\cdot)}(\mathbb{G}) }   \|\dchi_{B}\|_{L^{p'(\cdot)}(\mathbb{G}) }    \le C
%----------
\end{align*}
%--------
holds for all  balls $B\subset \mathbb{G}$.
%---------
 \label{enumerate:charact-norm-Lie-variable-dual}
%-------
\end{enumerate}
%-----------
\end{lemma}
%-------------

Hereafter, for a function $b$ defined on $\mathbb{G}$, we write
%-----------------
\begin{align*}
%-------
  b^{-}(x) :=- \min\{b, 0\} =
%-------
\begin{cases}
%------------
 0,  & \text{if}\ b(x) \ge 0  \\
 |b(x)|, & \text{if}\ b(x) < 0
%------------
\end{cases}
%----------
\end{align*}
%------------
and  $b^{+}(x) =|b(x)|-b^{-}(x)$. Obviously, $b(x)=b^{+}(x)-b^{-}(x)$.

%-----------------
\begin{lemma}\cite{wu2023some} \label{lem:non-negative-max-lip}
%---------------
 Let $0 <\beta <1$ and $b$ be a locally integrable function on $\mathbb{G}$. Then the following assertions are equivalent:
%--------------
\begin{enumerate}[label=(\roman*)]
%--------------
\item   $b\in  \Lambda_{\beta}(\mathbb{G})$  and $b\ge 0$.
 %-----------
     \label{enumerate:Lem-non-negative-max-lip-1}
%-------
   \item For all $1\le s<\infty$,  there exists a positive constant $C$ such that
%--------
\begin{align} \label{inequ:non-negative-max-lip}
%-----------
 \sup_{B} |B|^{-\beta/Q}  \left( |B|^{-1} \dint_{B}  |b(x) -M_{B}(b)(x)  |^{s} \mathd x \right)^{1/s} \le C.
%-----------------
\end{align}
%--------
    \label{enumerate:Lem-non-negative-max-lip-2}
%-----------
   \item \labelcref{inequ:non-negative-max-lip} holds for some $1\le s<\infty$.
%-----------
     \label{enumerate:Lem-non-negative-max-lip-3}
%-------
\end{enumerate}
%----------
\end{lemma}
%----------

Now we recall the Hardy-Littlewood-Sobolev inequality for the  fractional maximal function $\mathcal{M}_{\alpha}$ on Lie group.   The first part of \cref{lem:frac-max-Lie}  can be obtained from \cite{bernardis1994two} (see Theorem 1.6, or \cite{kokilashvili1989fractional,pan1992fractional,wheeden1993characterization},or Theorem A1 in \cite{guliyev2019characterizations}), and the second part comes from  Lemma 2.5 of \cite{liu2019multilinear} (or see theorem 1.1 in \cite{cruz2018boundedness}) in the framework of variable exponent Lebesgue spaces.
%  The Hardy–Littlewood–Sobolev theorem for the fractional integral operator \cite[Theorem 1]{pan1992fractional} \cite[Theorem 5.9.1]{bonfiglioli2007stratified}
% fractional maximal function 在无界域上 \cite{adamowicz2015maximal}
%-----------------
\begin{lemma}  \label{lem:frac-max-Lie}
%---------------
   Let $0 \le\alpha<Q$.
%--------------
\begin{enumerate}[label=(\arabic*),itemindent=2em]
%--------------
\item  If  $1/q = 1/p  -\alpha/Q$ with $1<p\le Q/\alpha$, then  $\mathcal{M}_{\alpha} $ is bounded from $L^{p}(\mathbb{G})$ to $L^{q}(\mathbb{G})$.
%-----------
   \label{enumerate:thm-A1-guliyev2019characterizations}
%-----------
 \item   If $p(\cdot)\in   \mathscr{C}^{\log}(\mathbb{G}) \cap \mathscr{P}(\mathbb{G})$   with $p_{+}\le Q/\alpha$, and define $q(\cdot)$ pointwise by
%-------------
\begin{align*}
%-------
   \dfrac{1}{q(x) }=\dfrac{1}{p(x)}- \dfrac{\alpha}{Q}  \qquad  \text{for} \ \almostevery~  x \in \mathbb{G}.
%----------
\end{align*}
%--------
 Then there exists a positive constant $C$ such that the inequality
%----------------
\begin{align*}
%-------
 \|\mathcal{M}_{\alpha}(f)\|_{q(\cdot)}\le C \|f \|_{p(\cdot)}
%----------
\end{align*}
%------------
holds for all  $f \in L^{p(\cdot)}(\mathbb{G})$.
%--------
    \label{enumerate:lem-5-liu2019multilinear} % \cite{liu2019multilinear}
%-----------
\end{enumerate}
%----------
%----------
\end{lemma}
%----------

By  \labelcref{enumerate:thm-A1-guliyev2019characterizations} of \cref{lem:frac-max-Lie}, if  $0 <\alpha<Q$,  $1<p< Q/\alpha$ and $f\in L^{p}(\mathbb{G})$, then $\mathcal{M}_{\alpha} (f)(x)<\infty $ almost everywhere. A similar result is also valid in variable Lebesgue spaces. And the method of proof can refer to Lemma 2.6 in \cite{zhang2019some}, so we omit its proof.
%-----------------
\begin{lemma}  \label{lem:frac-max-almost-every}
%---------------
   Let $0 <\alpha<Q$, $p(\cdot)\in \mathscr{P}(\mathbb{G})$ and $p_{+}<Q/\alpha$. If   $f\in L^{p(\cdot)}(\mathbb{G})$, then  $\mathcal{M}_{\alpha}(f)(x)<\infty $ for almost everywhere $x\in \mathbb{G}$.
%----------
\end{lemma}
%----------

Now, we give the following pointwise estimate for $[b,M_{\alpha}] $  on $\mathbb{G}$ when $b\in \Lambda_{\beta}(\mathbb{G})$ (see Lemma 2.10, \cite{wu2023some}).
%----------------
\begin{lemma} \label{lem:frac-maximal-pointwise}
Let   $0\le\alpha<Q$, $0<\beta <1$, $0<\alpha+\beta<Q$ and $f: \mathbb{G} \to \mathbb{R}$ be a locally integrable function.  If  $b\in \Lambda_{\beta}(\mathbb{G})$ and $b\ge 0$, then, for arbitrary   $x\in \mathbb{G} $ such that $M_{\alpha} (f)(x) <\infty$, we have
%------------
\begin{align*}
%-------
  \big|[b,M_{\alpha}] (f)(x)\big|  &\le \|b\|_{\Lambda_{\beta}(\mathbb{G})} M_{\alpha+\beta} (f)(x).
%----------
\end{align*}
%----------
\end{lemma}
%------------

The first two parts of the following results can be referred to    \cite[page 3331]{bastero2000commutators} or \cite{zhang2009commutators},   through elementary calculations and derivations, it is easy to check that the  assertions are true.
 The pointwise relations of the last part can be obtained from  Lemma 2.7 and Remark 6 in \cite{wu2023some} (also refer to \cite{bastero2000commutators}, \cite[Lemma 2.3]{zhang2009commutators} and \cite[Lemma 3.1]{guliyev2023characterizations}), %\cite[Lemma 3.2]{guliyev2022some} .    can be deduced by  elementary calculations.
%--------------------------------
\begin{lemma} \label{lem:frac-max-pointwise-assert}
%--------------------------
Let  $b$ be a locally integrable function on $\mathbb{G}$ and $B \subset \mathbb{G}$ be  an arbitrary  given    ball.
%-------
 \begin{enumerate}[label=(\roman*)]
%-------
\item  If $E=\{x\in B: b(x)\le b_{B}\}$ and $F=  B\setminus E =\{x\in B: b(x)> b_{B}\}$. Then the following equality
%----------
\begin{align*}
%-------
    \dint_{E} |b(x)-b_{B}| \mathd x  &=  \dint_{F} |b(x)-b_{B}| \mathd x
%----------
\end{align*}
%-------
  is trivially true.
%--------
    \label{enumerate:lem-2.13-1-wu2023some}
%------------
\item    Then for any  $x\in B$, we have
%---------------
\begin{align*}
%-------
    |b_{B}|   &\le  |B|^{-\alpha/Q} M_{\alpha,B}(b)(x).
%----------
\end{align*}
%--------
    \label{enumerate:lem-2.13-2-wu2023some}
%-------
\item If $0 \le \alpha<Q$.  Then, for all $x\in B$, we have
%----------
\begin{align*}
%------------
 M_{\alpha} (b\dchi_{B})(x)  &=  M_{\alpha,B}(b)(x)
%------------
 \\ \intertext{and}
%-----
    M_{\alpha} (\dchi_{B})(x)  &=  M_{\alpha,B}(\dchi_{B})(x)=|B|^{\alpha/Q}.
%------------
\end{align*}
%----------
 \label{enumerate:lem-2.11-wu2023some}
%--------
%----------
%----------
\end{enumerate}
%------------
\end{lemma}
%------------

%-----------------
\begin{lemma}  \label{lem:identity-homogeneous-variable-Lie}   % \refer to Proposition 2.18 in \cite{cruz2013variable}   (refer to \cite{cruz2013variable} for more correlated  information)
%---------------
  Given  $p(\cdot)\in   \mathscr{P}(\mathbb{G})$, then for all    $s>0$, we have
%----------------
\begin{align*}
%-------
 \||f|^{s}\|_{q(\cdot)}=  \|f \|_{sp(\cdot)}^{s}.
%----------
\end{align*}
%------------

%----------
\end{lemma}
%----------

%-----------------
\begin{proof}
The idea is taken from \cite[Proposition 2.18]{cruz2013variable}.  Let   $\lambda=\eta^{1/s}$, then,
it follows from \cref{def.variable-lebesgue-space-Lie} that
 %-------
\begin{align*}
%-------
  \big\| |f|^{s} \big\|_{ L^{p(\cdot)} (\mathbb{G})}  &=\inf \Big\{\eta>0:  \int_{\mathbb{G}} \Big( \frac{|f(x)|^{s}}{\eta} \Big)^{p(x)} \mathrm{d}x \le 1 \Big\} \\
  &= \inf \Big\{ \lambda^{s}>0:  \int_{\mathbb{G}} \Big( \frac{|f(x)| }{\lambda} \Big)^{sp(x)} \mathrm{d}x \le 1 \Big\} \\
  &= \|f\|_{ L^{sp(\cdot)} (\mathbb{G})}^{s}.
%-------
\tag*{\qedhere}
%------
\end{align*}
%------
\end{proof}
%---------

%------------------------
\section{Proofs of the main results} %Main results and their proofs} %  the principal results
\label{sec:result-proof}

Now we give the proofs of the \cref{thm:nonlinear-frac-max-lip-variable} and \cref{thm:frac-max-lip-variable}. %main results.

\subsection{Proof of \cref{thm:nonlinear-frac-max-lip-variable}}

In order to prove \cref{thm:nonlinear-frac-max-lip-variable}, we first prove the following lemma.

%-----------------------
\begin{lemma} \label{lem:frac-Lie-lip-variable-norm}
%-----------------------
Let  $0 <\beta <1$ and $0 <\alpha <Q$. If   $b$  is a locally integrable function on $\mathbb{G}$ and satisfies
%-----------
\begin{align} \label{inequ:lem-frac-Lie-lip-variable-norm}
%-----------
\sup_{B\subset \mathbb{G}} \dfrac{1}{|B|^{\beta/Q}} \dfrac{\Big\| \big(b -|B|^{-\alpha/Q}\mathcal{M}_{\alpha,B} (b) \big) \dchi_{B} \Big\|_{L^{s(\cdot)}(\mathbb{G}) }}{\|\dchi_{B}\|_{L^{s(\cdot)}(\mathbb{G}) }}  < \infty
%-----------
\end{align}
%-------
%-----------
 for some  $s(\cdot) \in   \mathscr{B} (\mathbb{G}) $, then $b\in \Lambda_{\beta}(\mathbb{G})$.
%-----------
\end{lemma}
%--------------------

%---------------
\begin{proof}
%--------------
Some ideas are taken from \cite{bastero2000commutators,zhang2009commutators,zhang2014commutators} and  \cite{zhang2019some}.

For any   $\mathbb{G}$-ball $B\subset \mathbb{G}$, let $E=\{x\in B: b(x)\le b_{B}\}$ and $F=  B\setminus E =\{x\in B: b(x)> b_{B}\}$.
Noticing from \cref{lem:frac-max-pointwise-assert}(ii) that
%----------
\begin{align*}
%------------------
%   \dint_{E} |b(x)-b_{B}| \mathd x  &=  \dint_{F} |b(x)-b_{B}| \mathd x
%%------------
%\\ \intertext{and}
%------------
    |b_{B}|   &\le  |B|^{-\alpha/Q} M_{\alpha,B}(b)(x) \qquad \forall ~ x\in B.
%------------
\end{align*}
%----------
Then, for any $x\in E\subset B$, we have $b(x)\le b_{B}\le  |b_{B}| \le   |B|^{-\alpha/Q} M_{\alpha,B}(b)(x)$.
It is clear that
%---------------
\begin{align*}
%-------
    |b(x)- b_{B}|   &\le \Big| b(x) - |B|^{-\alpha/Q} M_{\alpha,B}(b)(x) \Big|, \qquad  \forall~ x\in E.
%----------
\end{align*}
%-------

Therefore, by using  \cref{lem:frac-max-pointwise-assert}(i),   we get
%-----------------
\begin{align*}
%-------
 \dfrac{1}{|B|^{1+\beta/Q}} \dint_{B} \big| b(x)-b_{B}) \big| \mathd x &=  \dfrac{1}{|B|^{1+\beta/Q}} \dint_{E\cup F} \big| b(x)-b_{B}) \big| \mathd x  \\
 &= \dfrac{2}{|B|^{1+\beta/Q}} \dint_{E} \big| b(x)-b_{B}) \big| \mathd x    \\
  &\le \dfrac{2}{|B|^{1+\beta/Q}} \dint_{E} \Big| b(x) - |B|^{-\alpha/Q} M_{\alpha,B}(b)(x) \Big| \mathd x    \\
  &\le  \dfrac{2}{|B|^{1+\beta/Q}} \dint_{B} \Big| b(x) - |B|^{-\alpha/Q} M_{\alpha,B}(b)(x) \Big| \mathd x .
%----------
\end{align*}
%------------

By using \cref{lem:holder-inequality-Lie-group} \labelcref{enumerate:variable-holder-Lie-1}, \labelcref{inequ:lem-frac-Lie-lip-variable-norm} and   \cref{lem:norm-characteristic-Lie-liu2022characterisation}  \labelcref{enumerate:charact-norm-Lie-variable-dual}, we have
%-----------------
\begin{align*}
%-------
 \dfrac{1}{|B|^{1+\beta/Q}} & \dint_{B} \big| b(x)-b_{B}) \big| \mathd x   \\
  &\le  \dfrac{2}{|B|^{1+\beta/Q}} \dint_{B} \Big| b(x) - |B|^{-\alpha/Q} M_{\alpha,B}(b)(x) \Big| \mathd x  \\
  &\le \dfrac{C}{|B|^{1+\beta/Q}} \Big\| \big(b -|B|^{-\alpha/Q}\mathcal{M}_{\alpha,B} (b) \big) \dchi_{B} \Big\|_{L^{s(\cdot)}(\mathbb{G}) }   \|\dchi_{B}\|_{L^{s'(\cdot)}(\mathbb{G}) } \\
  &\le  \dfrac{C}{|B|}  \|\dchi_{B}\|_{L^{s(\cdot)}(\mathbb{G}) }   \|\dchi_{B}\|_{L^{s'(\cdot)}(\mathbb{G}) }    \le C.
%----------
\end{align*}
%------------

So, the proof is completed by applying \cref{lem:2.2-li2003lipschitz} and \cref{def.lip-space}.
%--------------------
\end{proof}
%---------

 Now, we show the mapping properties of the  nonlinear  commutator  $ [b,\mathcal{M}_{\alpha}] $ on variable Lebesgue spaces over some stratified Lie group $\mathbb{G}$  when the symbol   $b$ belongs to some Lipschitz space.

\begin{proof}[Proof of  \cref{thm:nonlinear-frac-max-lip-variable}]
%---------------
%Some ideas are taken from   \cite{zhang2019some}.
Since the implications \labelcref{enumerate:thm-nonlinear-frac-max-lip-2} $\xLongrightarrow{\ \  }$ \labelcref{enumerate:thm-nonlinear-frac-max-lip-3} and \labelcref{enumerate:thm-nonlinear-frac-max-lip-5} $\xLongrightarrow{\ \ }$ \labelcref{enumerate:thm-nonlinear-frac-max-lip-4} follow readily,
%and \labelcref{enumerate:thm-nonlinear-frac-max-lip-2} $\xLongrightarrow{\ \  }$ \labelcref{enumerate:thm-nonlinear-frac-max-lip-5} is similar to \labelcref{enumerate:thm-nonlinear-frac-max-lip-3} $\xLongrightarrow{\ \  }$ \labelcref{enumerate:thm-nonlinear-frac-max-lip-4},
we only need to prove \labelcref{enumerate:thm-nonlinear-frac-max-lip-1} $\xLongrightarrow{\ \  }$ \labelcref{enumerate:thm-nonlinear-frac-max-lip-2}, \labelcref{enumerate:thm-nonlinear-frac-max-lip-3} $\xLongrightarrow{\ \  }$ \labelcref{enumerate:thm-nonlinear-frac-max-lip-4},  \labelcref{enumerate:thm-nonlinear-frac-max-lip-4} $\xLongrightarrow{\ \  }$  \labelcref{enumerate:thm-nonlinear-frac-max-lip-1}
and \labelcref{enumerate:thm-nonlinear-frac-max-lip-2} $\xLongrightarrow{\ \  }$ \labelcref{enumerate:thm-nonlinear-frac-max-lip-5}.
%(see \Cref{fig:proof-structure-frac-max-lip-1} for the proof structure).

 \labelcref{enumerate:thm-nonlinear-frac-max-lip-1} $\xLongrightarrow{\ \  }$ \labelcref{enumerate:thm-nonlinear-frac-max-lip-2}:
 Let $b\in  \Lambda_{\beta}(\mathbb{G})$ and $b\ge 0$. We need to prove that $ [b,\mathcal{M}_{\alpha}] $ is bounded from $L^{p(\cdot)}(\mathbb{G})$ to $L^{q(\cdot)}(\mathbb{G})$ for all $(p(\cdot), q(\cdot))\in   \mathscr{B}^{\alpha+\beta}(\mathbb{G}) $. For such $p(\cdot)$ and any $f\in L^{p(\cdot)}(\mathbb{G})$, it
follows from \cref{lem:frac-max-almost-every} that $\mathcal{M}_{\alpha}(f)(x)<\infty $ for almost everywhere $x\in \mathbb{G}$. By using  \cref{lem:frac-maximal-pointwise}, we have
%----------
\begin{align*}
%-------------
  \big| [b,\mathcal{M}_{\alpha}](f)(x)  \big| \le \|b\|_{\Lambda_{\beta} (\mathbb{G})}  \mathcal{M}_{\alpha+\beta} (f)(x).
%--------------------
\end{align*}
%----------
 Then, statement \labelcref{enumerate:thm-nonlinear-frac-max-lip-2} follows from  \cref{lem:frac-max-Lie}\labelcref{enumerate:lem-5-liu2019multilinear} and \cref{lem.cor-2.4-variable-max-bounded}.

\labelcref{enumerate:thm-nonlinear-frac-max-lip-3} $\xLongrightarrow{\ \  }$ \labelcref{enumerate:thm-nonlinear-frac-max-lip-4}:
Let  $(p(\cdot), q(\cdot))\in   \mathscr{B}^{\alpha+\beta}(\mathbb{G}) $ be such that $ [b,\mathcal{M}_{\alpha}] $ is bounded from $L^{p(\cdot)}(\mathbb{G})$ to $L^{q(\cdot)}(\mathbb{G})$. We will verify \labelcref{inequ:thm-nonlinear-frac-max-lip-4} for $s(\cdot) =q(\cdot)$.

For any fixed   ball $B\subset \mathbb{G}$ and any  $x \in B$, it follows from   \cref{lem:frac-max-pointwise-assert} \labelcref{enumerate:lem-2.11-wu2023some} that
%----------
\begin{align*}
%------------------
   \mathcal{M}_{\alpha}(b\dchi_{B})(x)   =  \mathcal{M}_{\alpha,B}(b)(x)
%------------
 \ \text{and} \
%------------
    \mathcal{M}_{\alpha}(\dchi_{B})(x)   =  \mathcal{M}_{\alpha,B}(\dchi_{B})(x)=|B|^{\alpha/Q}.
%------------
\end{align*}
%----------
Then, for  any $x\in B$, we have
%---------------
\begin{align*}
%-------
    b(x)-|B|^{-\alpha/Q} \mathcal{M}_{\alpha,B} (b)(x)   &= |B|^{-\alpha/Q} \big( b(x) |B|^{\alpha/Q} -\mathcal{M}_{\alpha,B} (b)(x) \big)    \\
   &=    |B|^{-\alpha/Q} \big( b(x)  \mathcal{M}_{\alpha}(\dchi_{B})(x) -\mathcal{M}_{\alpha}(b\dchi_{B})(x) \big)     \\
    &=    |B|^{-\alpha/Q}  [b,\mathcal{M}_{\alpha}] (\dchi_{B})(x).
%----------
\end{align*}
%------------
 Thus, for any $x\in \mathbb{G}$, we get
%---------------
\begin{align*}
%-------
   \big( b(x)-|B|^{-\alpha/Q} \mathcal{M}_{\alpha,B} (b)(x) \big) \dchi_{B}(x) &=     |B|^{-\alpha/Q}  [b,\mathcal{M}_{\alpha}] (\dchi_{B})(x) \dchi_{B}(x).
%----------
\end{align*}
%------------
 Noticing from assertion  \labelcref{enumerate:thm-nonlinear-frac-max-lip-3} that $ [b,\mathcal{M}_{\alpha}] $ is bounded from $L^{p(\cdot)}(\mathbb{G})$ to $L^{q(\cdot)}(\mathbb{G})$ with  $1/q(\cdot) = 1/p(\cdot) -(\alpha+\beta)/Q$.  For any ball $B\subset \mathbb{G}$,  by   \cref{lem:norm-characteristic-Lie-liu2022characterisation}  \labelcref{enumerate:charact-norm-fraction-Lie-variable}, we have
%----------------
\begin{align*}
%-------
 \Big\| \big(b -|B|^{-\alpha/Q}\mathcal{M}_{\alpha,B} (b) \big) \dchi_{B} \Big\|_{L^{q(\cdot)}(\mathbb{G}) }
 &\le  |B|^{-\alpha/Q} \big\| [b,\mathcal{M}_{\alpha}] (\dchi_{B}) \big\|_{L^{q(\cdot)}(\mathbb{G}) }  \\
&\le C |B|^{-\alpha/Q} \big\|  \dchi_{B} \big\|_{L^{p(\cdot)}(\mathbb{G}) }  \\
 &\le C |B|^{-\alpha/Q}   |B|^{(\alpha+\beta)/Q}  \big\|  \dchi_{B} \big\|_{L^{q(\cdot)}(\mathbb{G}) }   \\
 &\le C   |B|^{ \beta/Q}  \big\|  \dchi_{B} \big\|_{L^{q(\cdot)}(\mathbb{G}) },
%----------
\end{align*}
%------------
which gives \labelcref{inequ:thm-nonlinear-frac-max-lip-4} since   $B$ is arbitrary and $C$ is independent of $B$.

\labelcref{enumerate:thm-nonlinear-frac-max-lip-4} $\xLongrightarrow{\ \  }$  \labelcref{enumerate:thm-nonlinear-frac-max-lip-1}:
By  \cref{lem:non-negative-max-lip}, it suffices to prove
%--------
\begin{align} \label{inequ:proof-lem-non-negative-max-lip-41}
%-----------
 \sup_{B} \dfrac{1}{|B|^{1+\beta/Q}} \dint_{B} \Big|b(y)- \mathcal{M}_{B} (b)(y) \Big|  \mathd y <\infty.
%-----------------
\end{align}
%--------

  For any fixed   ball $B\subset \mathbb{G}$, we have
%---------------
\begin{align} \label{inequ:proof-lem-non-negative-max-lip-41-1}
%-------
\begin{split}
%-------
  & \dfrac{1}{|B|^{1+\beta/Q}} \dint_{B} \Big|b(y)- \mathcal{M}_{B} (b)(y) \Big|  \mathd y  \\
    &\le  \dfrac{1}{|B|^{1+\beta/Q}} \dint_{B} \Big|b(y)- |B|^{-\alpha/Q}\mathcal{M}_{\alpha,B} (b) (y)  \Big|  \mathd y     \\
    &\qquad +  \dfrac{1}{|B|^{1+\beta/Q}} \dint_{B} \Big|  |B|^{-\alpha/Q}\mathcal{M}_{\alpha,B} (b) (y)-\mathcal{M}_{B} (b)(y) \Big|  \mathd y    \\
     &:=  I_{1}+I_{2}.
%-------
\end{split}
%----------
\end{align}
%-------
For $ I_{1}$, by applying statement \labelcref{enumerate:thm-nonlinear-frac-max-lip-4},   \cref{lem:holder-inequality-Lie-group} \labelcref{enumerate:variable-holder-Lie-1}  and   \cref{lem:norm-characteristic-Lie-liu2022characterisation}  \labelcref{enumerate:charact-norm-Lie-variable-dual}, we get
%---------------
\begin{align*}
%-------
 I_{1} &=   \dfrac{1}{|B|^{1+\beta/Q}} \dint_{B} \Big|b(y)- |B|^{-\alpha/Q}\mathcal{M}_{\alpha,B} (b) (y)  \Big|  \mathd y     \\
 &\le  \dfrac{C}{|B|^{1+\beta/Q}} \Big\| \big(b -|B|^{-\alpha/Q}\mathcal{M}_{\alpha,B} (b) \big) \dchi_{B} \Big\|_{L^{s(\cdot)}(\mathbb{G}) }  \|\dchi_{B}\|_{L^{s'(\cdot)}(\mathbb{G}) } \\
  &\le  \dfrac{C}{|B|^{\beta/Q}}\dfrac{ \Big\| \big(b -|B|^{-\alpha/Q}\mathcal{M}_{\alpha,B} (b) \big) \dchi_{B} \Big\|_{L^{s(\cdot)}(\mathbb{G}) }}{\|\dchi_{B}\|_{L^{s(\cdot)}(\mathbb{G}) }}   \\
&\le C,
%----------
\end{align*}
%-------
where the constant $C$ is independent of $B$.

 Now we consider $ I_{2}$. For all  $y\in B$,  it follows from \cref{lem:frac-max-pointwise-assert}  \labelcref{enumerate:lem-2.11-wu2023some} that
 %----------
\begin{align*}
%------------------
   \mathcal{M}_{\alpha}(\dchi_{B})(y)   =  |B|^{\alpha/Q} \ \text{and}\  \mathcal{M}_{\alpha}(b\dchi_{B})(y)   =  \mathcal{M}_{\alpha,B}(b)(y),
%------------
\\ \intertext{and}
%------------
    \mathcal{M}(\dchi_{B})(y)   = \dchi_{B}(y)   =  1 \ \text{and}\  \mathcal{M}(b\dchi_{B})(y)   =  \mathcal{M}_{B}(b)(y).
%------------
\end{align*}
%----------
Then, for any $y\in B$,  we get
%---------------
\begin{align} \label{inequ:proof-lem-non-negative-max-lip-41-2}
%-------
\begin{split}
%-------
  &\Big|  |B|^{-\alpha/Q}\mathcal{M}_{\alpha,B} (b) (y)-\mathcal{M}_{B} (b)(y) \Big|  \\
    &\le  |B|^{-\alpha/Q}  \Big|  \mathcal{M}_{\alpha,B} (b) (y)- |B|^{\alpha/Q} |b(y)|  \Big|  + \Big|  |b(y)|-\mathcal{M}_{B} (b)(y) \Big|    \\
    &\le  |B|^{-\alpha/Q}  \Big|   \mathcal{M}_{\alpha}(b\dchi_{B})(y)-  |b(y)| \mathcal{M}_{\alpha}(\dchi_{B})(y) \Big|  \\
    &\qquad + \Big|  |b(y)|\mathcal{M}(\dchi_{B})(y) -\mathcal{M}(b\dchi_{B})(y)  \Big|    \\
  &\le  |B|^{-\alpha/Q}  \Big|[|b|, \mathcal{M}_{\alpha}](\dchi_{B})(y) \Big|  + \Big| [ |b|, \mathcal{M}] (\dchi_{B})(y)   \Big|.
%----------
\end{split}
%-------
\end{align}
%-------

Since  $ s(\cdot)\in   \mathscr{B}(\mathbb{G}) $, statement \labelcref{enumerate:thm-nonlinear-frac-max-lip-4} along with \cref{lem:frac-Lie-lip-variable-norm} gives  $b\in  \Lambda_{\beta}(\mathbb{G})$, which implies $|b|\in  \Lambda_{\beta}(\mathbb{G})$.
Therefore, we can apply \cref{lem:frac-maximal-pointwise} to $[|b|, \mathcal{M}_{\alpha}]$ and $[ |b|, \mathcal{M}]$ due to  $|b|\in  \Lambda_{\beta}(\mathbb{G})$ and $|b|\ge0$.

By using  \cref{lem:frac-maximal-pointwise} and \cref{lem:frac-max-pointwise-assert}  \labelcref{enumerate:lem-2.11-wu2023some}, for any $y\in B$,  we have
%----------
\begin{align*}
%------------------
  \Big|[|b|, \mathcal{M}_{\alpha}](\dchi_{B})(y) \Big| &\le \|b\|_{\Lambda_{\beta} (\mathbb{G})}  \mathcal{M}_{\alpha+\beta} (\dchi_{B})(y) \le C \|b\|_{\Lambda_{\beta} (\mathbb{G})} |B|^{(\alpha+\beta)/Q}
%------------
\\ \intertext{and}
%------------
 \Big| [ |b|, \mathcal{M}] (\dchi_{B})(y)   \Big| &\le \|b\|_{\Lambda_{\beta} (\mathbb{G})}  \mathcal{M}_{\beta} (\dchi_{B})(y) \le C \|b\|_{\Lambda_{\beta} (\mathbb{G})} |B|^{\beta/Q}.
%------------
\end{align*}
%----------

Hence, it follows from  \labelcref{inequ:proof-lem-non-negative-max-lip-41-2} that
%---------------
\begin{align*}
%-------
 I_{2} &=   \dfrac{1}{|B|^{1+\beta/Q}} \dint_{B} \Big|  |B|^{-\alpha/Q}\mathcal{M}_{\alpha,B} (b) (y)-\mathcal{M}_{B} (b)(y) \Big|  \mathd y    \\
 &\le \dfrac{C}{|B|^{1+(\alpha+\beta)/Q}} \dint_{B}  \Big|[|b|, \mathcal{M}_{\alpha}](\dchi_{B})(y) \Big| \mathd y   +\dfrac{C}{|B|^{1+\beta/Q}} \dint_{B} \Big|  [ |b|, \mathcal{M}] (\dchi_{B})(y)   \Big|  \mathd y    \\
     &\le C \|b\|_{\Lambda_{\beta} (\mathbb{G})}.
%----------
\end{align*}
%-------

Putting the above estimates for $I_{1}$ and $I_{2}$ into \labelcref{inequ:proof-lem-non-negative-max-lip-41-1}, we obtain \labelcref{inequ:proof-lem-non-negative-max-lip-41}.

\labelcref{enumerate:thm-nonlinear-frac-max-lip-2} $\xLongrightarrow{\ \  }$ \labelcref{enumerate:thm-nonlinear-frac-max-lip-5}:
Assume statement \labelcref{enumerate:thm-nonlinear-frac-max-lip-2} is true. Reasoning as in the proof of \labelcref{enumerate:thm-nonlinear-frac-max-lip-3} $\xLongrightarrow{\ \  }$ \labelcref{enumerate:thm-nonlinear-frac-max-lip-4}, we have
%-----------
\begin{align} \label{inequ:proof-lem-non-negative-max-lip-25}
%-----------
\sup_{B\subset \mathbb{G}} \dfrac{1}{|B|^{\beta/Q}} \dfrac{\Big\| \big(b -|B|^{-\alpha/Q}\mathcal{M}_{\alpha,B} (b) \big) \dchi_{B} \Big\|_{L^{q(\cdot)}(\mathbb{G}) }}{\|\dchi_{B}\|_{L^{q(\cdot)}(\mathbb{G}) }}  < \infty
%-----------------
\end{align}
%-----------
holds for any $q(\cdot)$ for which there exists  a $p(\cdot)$ such that $(p(\cdot), q(\cdot))\in   \mathscr{B}^{\alpha+\beta}(\mathbb{G}) $.

On the other hand, for any    $s(\cdot) \in   \mathscr{B} (\mathbb{G}) $, choosing an $r>Q/(Q-\beta)>1$, it is easy to see from \cref{rem.variable-max-bounded} that  $rs(\cdot) (Q-\beta)/Q\in   \mathscr{B} (\mathbb{G}) $ and  $rs(\cdot) \in   \mathscr{B} (\mathbb{G}) $.
Set  $q(\cdot) =rs(\cdot)$ and define $p(\cdot)$  by $1/p(\cdot) = 1/q(\cdot)+(\alpha+\beta)/Q$. it is easy to verify that  $(p(\cdot), q(\cdot))\in   \mathscr{B}^{\alpha+\beta}(\mathbb{G}) $.

Noticing that
%----------------
\begin{align*}
%-------
  \frac{1}{s(\cdot)}=\frac{1}{r s(\cdot)}+ \frac{1}{r' s(\cdot)}  =\frac{1}{q(\cdot)}+ \frac{1}{r' s(\cdot)}.
%----------
\end{align*}
%---------
Hence, it follows from \cref{lem:holder-inequality-Lie-group} \labelcref{enumerate:variable-holder-Lie}, \labelcref{inequ:proof-lem-non-negative-max-lip-25} and  \cref{lem:identity-homogeneous-variable-Lie}  that
%-----------
\begin{align*}
%-----------
  \dfrac{1}{|B|^{\beta/Q}} & \dfrac{\Big\| \big(b -|B|^{-\alpha/Q}\mathcal{M}_{\alpha,B} (b) \big) \dchi_{B} \Big\|_{L^{s(\cdot)}(\mathbb{G}) }}{\|\dchi_{B}\|_{L^{s(\cdot)}(\mathbb{G}) }}    \\
  &\le   \dfrac{1}{|B|^{\beta/Q}}   \dfrac{\Big\| \big(b -|B|^{-\alpha/Q}\mathcal{M}_{\alpha,B} (b) \big) \dchi_{B} \Big\|_{L^{q(\cdot)}(\mathbb{G}) }  \| \dchi_{B} \|_{L^{r's(\cdot)}(\mathbb{G}) } }{\|\dchi_{B}\|_{L^{s(\cdot)}(\mathbb{G}) }}       \\
  &\le      \dfrac{C\| \dchi_{B}  \|_{L^{q(\cdot)}(\mathbb{G}) }  \| \dchi_{B} \|_{L^{r's(\cdot)}(\mathbb{G}) } }{\|\dchi_{B}\|_{L^{s(\cdot)}(\mathbb{G}) }}   \\
  &=   \dfrac{C\| \dchi_{B}  \|_{L^{s(\cdot)}(\mathbb{G}) }^{1/r}  \| \dchi_{B} \|_{L^{s(\cdot)}(\mathbb{G}) }^{1/r'} }{\|\dchi_{B}\|_{L^{s(\cdot)}(\mathbb{G}) }} =C,
%-----------------
\end{align*}
%--------
which gives \labelcref{inequ:thm-nonlinear-frac-max-lip-4} since $B$ is arbitrary and $C$ is independent of $B$.

%---------------
 This completes the proof of \cref{thm:nonlinear-frac-max-lip-variable}.
%--------------------------------
\end{proof}
%--------------

%--------------------
\begin{remark}   \label{rem.variable-max-bounded}
%-------------
 The proof of  \labelcref{enumerate:thm-nonlinear-frac-max-lip-2} $\xLongrightarrow{\ \  }$ \labelcref{enumerate:thm-nonlinear-frac-max-lip-5} is also valid for  $ \beta=0$.
%------------
\end{remark}
%----------

\subsection{Proof of \cref{thm:frac-max-lip-variable}}

Now, we give the proof of necessary and sufficient conditions for the boundedness of the maximal commutator $\mathcal{M}_{\alpha,b}$ on variable Lebesgue spaces in the context of  stratified Lie group when $b$ belongs to a Lipschitz space.

%-----------------
\begin{proof}[Proof of  \cref{thm:frac-max-lip-variable}]
%---------------
Similar to the proof of  \cref{thm:nonlinear-frac-max-lip-variable},
%Since the implications \labelcref{enumerate:thm-frac-max-lip-2} $\xLongrightarrow{\ \  }$ \labelcref{enumerate:thm-frac-max-lip-3} and \labelcref{enumerate:thm-frac-max-lip-5} $\xLongrightarrow{\ \ }$ \labelcref{enumerate:thm-frac-max-lip-4} follow readily,
we only need to prove \labelcref{enumerate:thm-frac-max-lip-1} $\xLongrightarrow{\ \  }$ \labelcref{enumerate:thm-frac-max-lip-2}, \labelcref{enumerate:thm-frac-max-lip-3} $\xLongrightarrow{\ \  }$ \labelcref{enumerate:thm-frac-max-lip-4},  \labelcref{enumerate:thm-frac-max-lip-4} $\xLongrightarrow{\ \  }$  \labelcref{enumerate:thm-nonlinear-frac-max-lip-1}
and \labelcref{enumerate:thm-frac-max-lip-2} $\xLongrightarrow{\ \  }$ \labelcref{enumerate:thm-frac-max-lip-5}.
%(the proof structure is also shown in \Cref{fig:proof-structure-frac-max-lip-1}).

 \labelcref{enumerate:thm-frac-max-lip-1} $\xLongrightarrow{\ \  }$ \labelcref{enumerate:thm-frac-max-lip-2}:
If $b\in  \Lambda_{\beta}(\mathbb{G})$, then, for any ball $B\subset \mathbb{G}$ containing $x$ and $y$, using \cref{lem:2.2-li2003lipschitz} \labelcref{enumerate:property-lip-Lie-1,enumerate:property-lip-Lie-3}, we have
%----------
\begin{align*}
%-------
 \mathcal{M}_{\alpha,b} (f)(x) &= \sup_{B\ni x \atop B\subset \mathbb{G}}  \dfrac{1}{|B|^{1-\alpha/Q}} \dint_{B} |b(x)-b(y)| |f(y)| \mathd y   \\
 &\le C  \|b\|_{\Lambda_{\beta}(\mathbb{G})}  \sup_{B\ni x \atop B\subset \mathbb{G}} \dfrac{1}{|B|^{1-(\alpha+\beta)/Q}} \dint_{B}   |f(y)| \mathd y   \\
 &\le C  \|b\|_{\Lambda_{\beta}(\mathbb{G})}  \mathcal{M}_{\alpha+\beta} (f)(x).
%----------
\end{align*}
%--------

This, together with \cref{lem:frac-max-Lie} \labelcref{enumerate:lem-5-liu2019multilinear}, shows that $\mathcal{M}_{\alpha,b}$ is bounded from  $ L^{p(\cdot)}(\mathbb{G})$ to  $ L^{q(\cdot)}(\mathbb{G})$.

%----------
\labelcref{enumerate:thm-frac-max-lip-3} $\xLongrightarrow{\ \  }$ \labelcref{enumerate:thm-frac-max-lip-4}:
Let  $(p(\cdot), q(\cdot))\in   \mathscr{B}^{\alpha+\beta}(\mathbb{G}) $ be such that $\mathcal{M}_{\alpha,b}$  is bounded from $L^{p(\cdot)}(\mathbb{G})$ to $L^{q(\cdot)}(\mathbb{G})$. We will verify \labelcref{inequ:thm-frac-max-lip-4} for $s(\cdot) =q(\cdot)$.

For any fixed   ball $B\subset \mathbb{G}$ and any  $x \in B$,   we have
%---------------
\begin{align*}
%-------
    \big| b(x)-b_{B}\big|    &\le |B|^{-1} \dint_{B} \big| b(x)  - b(y) \big| \mathd y   \\
   &=    |B|^{-1} \dint_{B} \big| b(x)  - b(y) \big| \dchi_{B} (y)\mathd y     \\
    &\le    |B|^{-\alpha/Q}  \mathcal{M}_{\alpha,b}  (\dchi_{B})(x).
%----------
\end{align*}
%------------
 Then, for all $x\in \mathbb{G}$, we get
%---------------
\begin{align*}
%-------
  \Big| \big( b(x)-b_{B} \big) \dchi_{B}(x)\Big| \le    |B|^{-\alpha/Q}  \mathcal{M}_{\alpha,b}  (\dchi_{B})(x).
%----------
\end{align*}
%------------
Since $\mathcal{M}_{\alpha,b}$ is bounded from  $ L^{p(\cdot)}(\mathbb{G})$ to  $ L^{q(\cdot)}(\mathbb{G})$, by \cref{lem:norm-characteristic-Lie-liu2022characterisation} \labelcref{enumerate:charact-norm-fraction-Lie-variable}, we obtain
%----------------
\begin{align*}
%-------
 \Big\| \big( b -b_{B} \big) \dchi_{B} \Big\|_{L^{q(\cdot)}(\mathbb{G}) }
 &\le  |B|^{-\alpha/Q} \big\| \mathcal{M}_{\alpha,b}  (\dchi_{B}) \big\|_{L^{q(\cdot)}(\mathbb{G}) }  \\
&\le C |B|^{-\alpha/Q} \big\|  \dchi_{B} \big\|_{L^{p(\cdot)}(\mathbb{G}) }  \\
% &\le C |B|^{-\alpha/Q}   |B|^{(\alpha+\beta)/Q}  \big\|  \dchi_{B} \big\|_{L^{q(\cdot)}(\mathbb{G}) }   \\
 &\le C   |B|^{ \beta/Q}  \big\|  \dchi_{B} \big\|_{L^{q(\cdot)}(\mathbb{G}) },
%----------
\end{align*}
%------------
which gives \labelcref{inequ:thm-frac-max-lip-4} for $s(\cdot) =q(\cdot)$ since   $B$ is arbitrary and $C$ is independent of $B$.

\labelcref{enumerate:thm-frac-max-lip-4} $\xLongrightarrow{\ \  }$  \labelcref{enumerate:thm-frac-max-lip-1}:
 For any    ball $B\subset \mathbb{G}$,  by using  generalized H\"{o}lder's inequality \cref{lem:holder-inequality-Lie-group}\labelcref{enumerate:variable-holder-Lie-1},  assertion  \labelcref{enumerate:thm-frac-max-lip-4} and  \cref{lem:norm-characteristic-Lie-liu2022characterisation} \labelcref{enumerate:charact-norm-Lie-variable-dual}, we obtain
%-----------------
\begin{align*}
%-------
 \dfrac{1}{|B|^{1+\beta/Q}} \dint_{B} \big| b(y)-b_{B}) \big| \mathd y
% &=   \dfrac{1}{|B|^{1+\beta/Q}} \dint_{B} \big| b(y)-b_{B}) \big| \dchi_{B}(y) \mathd y \\
  &\le \dfrac{C}{|B|^{1+\beta/Q}} \big\| \big(b-b_{B} \big)\dchi_{B}  \big\|_{L^{q(\cdot)}(\mathbb{G}) }   \|\dchi_{B}  \|_{L^{q'(\cdot)}(\mathbb{G})}   \\
   &=  \dfrac{C}{|B|^{\beta/Q}} \dfrac{\big\| \big(b-b_{B} \big)\dchi_{B}  \big\|_{L^{q(\cdot)}(\mathbb{G}) }}{\|\dchi_{B}  \|_{L^{q(\cdot)}(\mathbb{G})} }       \\
  &\qquad \times    \dfrac{1}{|B|}  \|\dchi_{B}  \|_{L^{q(\cdot)}(\mathbb{G})} \|\dchi_{B}  \|_{L^{q'(\cdot)}(\mathbb{G})}    \\
&\le  \dfrac{C}{|B|^{\beta/Q}} \dfrac{\big\| \big(b-b_{B} \big)\dchi_{B}  \big\|_{L^{q(\cdot)}(\mathbb{G}) }}{\|\dchi_{B}  \|_{L^{q(\cdot)}(\mathbb{G})} }       \\
&\le C.
%----------
\end{align*}
%------------

 This shows that $b\in  \Lambda_{\beta}(\mathbb{G})$ by \cref{lem:2.2-li2003lipschitz} \labelcref{enumerate:property-lip-Lie-1} and  \cref{def.lip-space} since the constant $C$ is independent of $B$.

%-------
\labelcref{enumerate:thm-frac-max-lip-2} $\xLongrightarrow{\ \  }$ \labelcref{enumerate:thm-frac-max-lip-5}:
Similar to the proof course of   \cref{thm:nonlinear-frac-max-lip-variable}.
Assume statement \labelcref{enumerate:thm-frac-max-lip-2} is true. Reasoning as in the proof of \labelcref{enumerate:thm-frac-max-lip-3} $\xLongrightarrow{\ \  }$ \labelcref{enumerate:thm-frac-max-lip-4}, we have
%-----------
\begin{align} \label{inequ:proof-frac-max-lip-25}
%-----------
\sup_{B\subset \mathbb{G}} \dfrac{1}{|B|^{\beta/Q}} \dfrac{\Big\| \big(b -b_{B}   \big) \dchi_{B} \Big\|_{L^{q(\cdot)}(\mathbb{G}) }}{\|\dchi_{B}\|_{L^{q(\cdot)}(\mathbb{G}) }}  &< \infty
%-----------------
\end{align}
%-----------
holds for any $q(\cdot)$ for which there exists  a $p(\cdot)$ such that $(p(\cdot), q(\cdot))\in   \mathscr{B}^{\alpha+\beta}(\mathbb{G}) $.

%On the other hand, for any    $s(\cdot) \in   \mathscr{B} (\mathbb{G}) $, choosing an $r>Q/(Q-\beta)>1$, it is easy to see from \cref{rem.variable-max-bounded} that  $rs(\cdot) (Q-\beta)/Q\in   \mathscr{B} (\mathbb{G}) $ and  $rs(\cdot) \in   \mathscr{B} (\mathbb{G}) $.
%Set  $q(\cdot) =rs(\cdot)$ and define $p(\cdot)$  by $1/p(\cdot) = 1/q(\cdot)+(\alpha+\beta)/Q$. it is easy to verify that  $(p(\cdot), q(\cdot))\in   \mathscr{B}^{\alpha+\beta}(\mathbb{G}) $.

On the other hand, let $s(\cdot)$, $q(\cdot)$ and $r$ be given in \labelcref{enumerate:thm-nonlinear-frac-max-lip-2} $\xLongrightarrow{\ \  }$ \labelcref{enumerate:thm-nonlinear-frac-max-lip-5} of \cref{thm:nonlinear-frac-max-lip-variable} and satisfy
%----------------
\begin{align*}
%-------
  \frac{1}{s(\cdot)}=\frac{1}{r s(\cdot)}+ \frac{1}{r' s(\cdot)}  =\frac{1}{q(\cdot)}+ \frac{1}{r' s(\cdot)}.
%----------
\end{align*}
%---------
Thus, it follows from \cref{lem:holder-inequality-Lie-group} \labelcref{enumerate:variable-holder-Lie}, \labelcref{inequ:proof-frac-max-lip-25} and  \cref{lem:identity-homogeneous-variable-Lie}  that
%-----------
\begin{align*}
%-----------
  \dfrac{1}{|B|^{\beta/Q}} \dfrac{\Big\| \big(b -b_{B}   \big) \dchi_{B} \Big\|_{L^{s(\cdot)}(\mathbb{G}) }}{\|\dchi_{B}\|_{L^{s(\cdot)}(\mathbb{G}) }}
  &\le   \dfrac{1}{|B|^{\beta/Q}}   \dfrac{\Big\| \big(b -b_{B}   \big) \dchi_{B} \Big\|_{L^{q(\cdot)}(\mathbb{G}) }  \| \dchi_{B} \|_{L^{r's(\cdot)}(\mathbb{G}) } }{\|\dchi_{B}\|_{L^{s(\cdot)}(\mathbb{G}) }}       \\
  &\le      \dfrac{C\| \dchi_{B}  \|_{L^{q(\cdot)}(\mathbb{G}) }  \| \dchi_{B} \|_{L^{r's(\cdot)}(\mathbb{G}) } }{\|\dchi_{B}\|_{L^{s(\cdot)}(\mathbb{G}) }}   \\
  &=   \dfrac{C\| \dchi_{B}  \|_{L^{s(\cdot)}(\mathbb{G}) }^{1/r}  \| \dchi_{B} \|_{L^{s(\cdot)}(\mathbb{G}) }^{1/r'} }{\|\dchi_{B}\|_{L^{s(\cdot)}(\mathbb{G}) }} =C,
%-----------------
\end{align*}
%--------
which gives  \labelcref{inequ:thm-frac-max-lip-4} since $B$ is arbitrary and $C$ is independent of $B$.

%---------------
 The  proof of \cref{thm:frac-max-lip-variable} is finished.
%--------------------------------
\end{proof}
%--------------

%------------

%-----------------------
% \subsubsection*{Acknowledgments:}
%The authors cordially  thank the anonymous referees who gave valuable  suggestions and useful comments which have lead to the improvement of this paper.

%-----------------------
 \subsubsection*{Funding information:}
 Zhao is  financially supported by the  Scientific Research Fund of AHPU (No.S022022177).
Wu is supported in parts by the NNSF of China (No.11571160),  the Science and Technology Fund of Heilongjiang  (No.2019-KYYWF-0909),   the Reform and Development Foundation for Local Colleges and Universities of the Central Government (No.2020YQ07) and the Scientific Research Fund of MNU (No.D211220637).
% This work was partly supported by the Scientific Research Fund of APU (No.S022022177, for Zhao), Project of Heilongjiang Province Science and Technology Program (No.2019-KYYWF-0909, for Wu), the National Natural Science Foundation of China (No.11571160, for Wu), the Reform and Development Foundation for Local Colleges and Universities of the Central Government(No.2020YQ07, for Wu) and the Scientific Research Fund of Mudanjiang Normal University (No.D211220637, for Wu).

%-----------------------
 \subsubsection*{Conflict of interest: }
The authors state that there is no conflict of interest.   %Authors state no conflict of  interest.

%-----------------------
 \subsubsection*{Data availability statement:}
  All data generated or analysed during this study are included in this published article.
%-----------------------
 \subsubsection*{Author contributions:}
 All authors contributed equally to the writing of this article.
 All authors read the final manuscript and approved its submission.

%%-----------------------
%\textbf{Disclosure statement:} Authors state no conflict of  interest.
%
%\textbf{Acknowledgments:} The authors cordially  thank the anonymous referees who gave valuable  suggestions and useful comments which have lead to the improvement of this paper.
%
%\textbf{Funding information:} This work was partly supported by the National Natural Science Foundation of China
%(Grant No. 11571160), Scientific Project-HLJ (No.2019-KYYWF-0909,1355ZD010), and the  Reform and Development Foundation for Local Colleges and Universities of the
%Central Government(No.2020YQ07).
%
%%-----------------------

%------------

%============================
%\section*{References}
%---------------------------------
\phantomsection
\addcontentsline{toc}{section}{References}
%----------
\bibliographystyle{tugboat}          %ieeetr 参考文献 编号 是 数字，按字母顺序排列 jIEEEtran 按引用
                               %plain，按字母的顺序排列，比较次序为作者、年度和标题. plainnat 全称  plplain
                               %unsrt，样式同plain，只是按照引用的先后排序.  unsrtnat  achemso-年加粗
                               %abbrv，类似plain，将月份全拼改为缩写，更显紧凑.
                               %apalike  姓前，名缩写 年代在前 按字母排序 chicago,  munich, asmejour  erae
                               %siam，美国工业和应用数学学会期刊样式.  按字母排序 年在后  共同作者省略  aomalpha  ier  tugboat
                               %acm，美国计算机学会期刊样式.  大写 姓前，名缩写 按字母排序 年在后  ACM-Reference-Format 全称小写 年在前后
                               % elsarticle-num-names 姓后 缩写名
                               % elsarticle-harv 姓前 缩写名  按字母排序   年在前
                               % elsarticle-num 姓后 缩写名 按引用   asmeconf 全称
%----------------------------
\bibliography{wu-reference}

%============================
\end{document}